\documentclass[10pt]{amsart}

\usepackage{amssymb}
\usepackage{tikz}
\usepackage{epsfig}   
\usepackage[foot]{amsaddr}
\usepackage{caption}
\usepackage{enumerate}

\usepackage[T1]{fontenc}

\usepackage{setspace}
\usepackage{parskip}
\newtheorem{thm}{Theorem}[section]
\newtheorem{prop}[thm]{Proposition}
\newtheorem{lem}[thm]{Lemma}
\newtheorem{cor}[thm]{Corollary}

\theoremstyle{definition}
\newtheorem{definition}[thm]{Definition}

\newtheorem{claim}[thm]{Claim}

\theoremstyle{remark}

\numberwithin{equation}{section}

\newcommand{\R}{\mathbb{R}}  
\newcommand{\Q}{\mathbb{Q}}  
\newcommand{\str}{\textrm{str}} 
\newcommand{\rev}{\textrm{rev}} 
\newcommand{\Z}{\mathbb{Z}}
\newcommand{\N}{\mathbb{N}}
\newcommand{\K}{\mathcal{K}}
\newcommand{\bK}{\mathbb{K}}
\newcommand{\bL}{\mathbb{L}}
\newcommand{\bM}{\mathbb{M}}

\newcommand{\norm}[1]{\left\lvert#1\right\rvert}

\newcommand{\Homeo}{\textrm{Homeo}}  
\newcommand{\Epi}{\textrm{Hom}}  
\newcommand{\aut}{\textrm{Aut}}
\newcommand{\dsup}{d_{\textrm{sup}}}
\newcommand{\fl}[1]{\left\langle#1\right\rangle}

\newcommand\blfootnote[1]{%
  \begingroup
  \renewcommand\thefootnote{}\footnote{#1}%
  \addtocounter{footnote}{-1}%
  \endgroup
}

\begin{document}

\title[homeomorphism group of universal Knaster continuum]{The homeomorphism group of the universal Knaster continuum}

\author{Sumun Iyer}
\address{Cornell University}
\email{ssi22@cornell.edu}
\blfootnote{Research supported by NSF GRFP grant DGE – 2139899 and NSF grant DMS-1954069}

\begin{abstract}
This paper has two aims: the first is to define a projective Fraissé family whose limit approximates the universal Knaster continuum. The family is such that the group $\aut(\mathbb{K})$ of automorphisms of the Fraissé limit is a dense subgroup of the group, $\Homeo(K)$, of homeomorphisms of the universal Knaster continuum.

The second aim is to compute the universal minimal flows of $\aut(\mathbb{K})$ and $\Homeo(K)$. We prove that both have universal minimal flow homeomorphic to the universal minimal flow of the free abelian group on countably many generators. The computation involves proving that both groups contain an open, normal subgroup which is extremely amenable.
\end{abstract}

\maketitle

\bigskip



\section{Introduction}
Knaster continua are a well-studied class of indecomposable continua (precise definitions are in Section 2). We focus here on the universal Knaster continuum--a Knaster continuum which continuously and openly surjects onto all other Knaster continua and particularly on the dynamics of the Polish group $\Homeo(K)$ of homeomorphisms of the universal Knaster continuum. Some recent results on the dynamics of continua homeomorphism groups of the sort we are interested in here can be found in \cite{bkumf} and \cite{pestov}. In each case, an appropriate \emph{extremely amenable} subgroup of the homeomorphism group is found and studied. A topological group is \emph{extremely amenable} when every continuous action of it on a compact Hausdorff space has a fixed point (see \cite{pestovbook} for more on the topic). 

The main result of the current paper is a structural theorem about the group $\Homeo (K)$:
\begin{thm}\label{thmintro}
The group $\Homeo (K)$ is isomorphic as a topological group to $U \rtimes F$ where $U$ is an extremely amenable Polish group and $F$ is the free abelian group on countably many generators with the discrete topology.
\end{thm} 
\noindent Recall that the topology on the semidirect product of two topological groups is the product topology on the underlying product space. Theorem \ref{thmintro} is proven in Section \ref{sec7degreeone} (as Theorem \ref{thm5.5}). We also show (Proposition \ref{prop5}) that $\Homeo(K)$ is \emph{not} isomorphic to the direct product $U \times F$. 

A few definitions before we say more about the motivation for Theorem \ref{thmintro}; if $G$ is a topological group, we call a compact Hausdorff space $X$ equipped with a continuous $G$-action, a \emph{$G$-flow}. It is \emph{minimal} if the orbit of each point is dense in $X$. By abstract topological dynamics, for each topological group $G$ we can find a unique (up to isomorphism) flow $\mathcal{M}(G)$ that is minimal and has the property that for any minimal $G$-flow $X$, there is a $G$-equivariant, continuous surjection $\mathcal{M}(G) \to X$ (see \cite{uspchains} for a short proof of these facts). The flow $\mathcal{M}(G)$ is called the \emph{universal minimal flow} of $G$. Observe that $G$ is extremely amenable if and only if $\mathcal{M}(G)$ is a singleton. After extreme amenability the next dividing line in complexity is whether or not $\mathcal{M}(G)$ is a metrizable space.
%In general, the theory of metrizable universal minimal flows is well-developed and many non-trivial examples have been computed-- see \cite{kpt}, \cite{glasnerweiss}, \cite{bkumf}. 
By results of Ben Yaacov, Melleray, and Tsankov in \cite{bymt}, when $\mathcal{M}(G)$ is metrizable, $G$ must contain a large (closed, co-precompact) extremely amenable subgroup $H$ and the universal minimal flow of $G$ is the translation action on the completion of $G/H$. Our case is firmly outside the realm of \cite{bymt}; it is immediate from Theorem \ref{thmintro} that $\mathcal{M}(\Homeo(K))$ is non-metrizable and with a little more work that $\mathcal{M}(\Homeo(K))$ is homeomorphic to $\mathcal{M}(F)$ (Corollary \ref{cor4}). Even so it is interesting to compare the situation here with the theory of metrizable universal minimal flows developed in \cite{bymt}. The group $\Homeo(K)$ also has a ``large''--in that it is open, normal (it is of course, \emph{not} pre-compact)-- extremely amenable subgroup whose quotient determines the behavior of its universal minimal flow. 

The main tools used in the paper are \emph{projective Fraissé families} and the Kechris-Pestov-Todorcevic correspondence. Irwin and Solecki in \cite{irw-sol} introduced projective Fraissé families of finite structures, a dual version of a classical construction from model theory. The Fraissé limits in the dual setting carry zero-dimensional compact topologies (rather than being countable spaces as in the classical setting). Under mild assumptions, the projective Fraissé family will have a \emph{canonical quotient} that is a compact and often connected topological space (precise definitions are in Section \ref{sec3fraisecategory}). Projective Fraissé families are by now a well-used tool for studying continua, they are used for example in the work of Bartošovà and Kwiatkowska in computing the universal minimal flow of the homeomorphism group of the Lelek fan (see \cite{lelekfan}, \cite{bkumf}). The Kechris-Pestov-Todorcevic correspondence is a general result that connects extreme amenability of automorphism groups of Fraissé limits with the Ramsey property for classes of finite structures (see \cite{kpt}).

Now we outline the structure of the paper. Section 2 contains some background on Knaster continua. In Section \ref{sec3fraisecategory}, we define a family $\mathcal{K}$ of finite structures and prove that it is a projective Fraissé family (Theorem \ref{thm1}). In Section \ref{sec4fraisselimit}, we prove that the quotient of the Fraissé limit, $\bK$, is homeomorphic to the universal Knaster continuum, $K$. In Section \ref{sec5density}, we show that the group $\aut(\bK)$ embeds densely into the group $\Homeo(K)$ via an approximate projective homogeneity property for $K$ (Theorem \ref{thm4.5}). We note that Wickman in her Ph.D. thesis,  \cite{wickman}, independently discovered a projective Fraissé class whose limit approximates the universal Knaster continuum. The automorphism group of the Fraissé limit of the class from \cite{wickman} however does not embed densely into the group $\Homeo (K)$ which is the main property we are concerned with in this paper.

In Section 6, we collect some results of D\k{e}bski that we need for the computation of universal minimal flows. The central notion needed is D\k{e}bski's definition of \emph{degree} for homeomorphisms of the Knaster continuum. In Section \ref{sec7degreeone}, we prove the main theorem (Theorem \ref{thm5.5}). 

\subsection*{Acknowledgements} I would like to thank S\l awek Solecki and Lauren Wickman for helpful conversations about this project. I would also like to thank S\l awek Solecki for suggesting this question to work on.

\section{Background on Knaster continua}\label{sec2knaster}
Let $I=[0,1]$ be the closed unit interval. A \emph{Knaster continuum} is a continuum of the form $\varprojlim (I, T_n^{n+1})$ where each $T_n^{n+1}: I \to I$ is continuous, open, surjection that maps $0$ to $0$. In the case that all but finitely many of the maps $T_n^{n+1}$ are monotone, the resulting continuum $\varprojlim (I, T_n^{n+1})$ is homeomorphic to an arc. Except for this trivial case, all other Knaster continua are \emph{indecomposable}--they cannot be written as the union of two non-trivial proper subcontinua. The simplest example of a Knaster continuum is one in which each bonding map is the tent-map given by

\[
f_2(x)=\begin{cases}
2x & \textrm{ if }x\leq \frac{1}{2}\\
2-2x & \textrm{ if }x>\frac{1}{2}\\
\end{cases}
\]

\noindent and is known as Brouwer's ``buckethandle" continuum. Indeed, it can be embedded into $\R^2$ in such a way that it resembles a thick buckethandle whose sections are Cantor sets (several pictures of Knaster continua can be found in \cite{watkins}). We collect below a few basic facts and definitions about Knaster continua that we will need.

First, an observation about continuous open maps $I \to I$ (it can be found as \cite{eberhart}, Theorem 2.1 for reference).

\begin{lem}\label{obs0}
Let $f:I \to I$ be open and continuous, with $f(0)=0$. Then, there exists $0=x_0<x_1<\cdots<x_n=1$ so that $f \restriction_{[x_i,x_{i+1}]}$ is a homeomorphism of $[x_i,x_{i+1}]$ onto $[0,1]$ and $f(x_i)=i \pmod 2$ for all $i$.
\end{lem}

Given $f$ as in Lemma \ref{obs0}, the number $n$ is unique and is called the \emph{degree} of $f$. We denote it by $\deg (f)$. An obvious but somehow vital observation is that if $f$ and $g$ are open continuous maps $I \to I$, then $\deg (f \circ g)=\deg (f) \deg (g)$.

Let $K=\varprojlim (I_n, T_n^{n+1})$ be a Knaster continuum. The topology on $K$ is the product topology and it is induced by the metric:
\[d_K(x,y)=\sum_{n=0}^\infty \frac{\norm{x(n)-y(n)}}{2^{n+1}}\]
for all $x,y \in K$.
We always consider the group of all homeomorphisms of $K$, which we denote $\Homeo(K)$, with the topology of uniform convergence. This is the topology induced by the metric:
\[d_{\sup}(f,g)=\sup_{x \in K} d_K(f(x),g(x))\]
for all $f,g \in \Homeo(K)$.
With this topology, $\Homeo(K)$ is a Polish group (a topological group whose topology is completely metrizable and separable).

\section{A Fraisse category of finite linear graphs}\label{sec3fraisecategory}

We define in this section a projective Fraissé category of finite graphs to approximate the universal Knaster continuum. A \emph{finite linear graph}, $(V,R)$, is a finite set $V$ of vertices together with an edge relation $R \subseteq V^2$ satisfying:
\begin{enumerate}
    \item for each $a, b \in V$, $(a,a) \in R$ and $(a,b) \in R \implies (b,a) \in R$
    \item for each $a \in V$, $\norm{N_R(a)} \in \{1,2\}$ where $N_R(a):=\{b \in V \setminus \{a\} \ : \ (a,b) \in R\}$
    \item $\norm{\{a\in V \ : \ \norm{N_R(a)}=1\}}=2$
\end{enumerate}
Note that our graphs are different than usual combinatorial graphs in that we allow (in fact, require) loops at vertices.

These two special vertices that satisfy condition (3) will be called \emph{end vertices of V}. A subset of a finite linear graph is \emph{connected} if it is connected with respect to the edge relation. Also we have the usual graph metric, $d_R$, on $(V,R)$ where $d_R(x,y)$ is the length of the shortest $R$-path in $V$ connecting $x$ to $y$. 

A \emph{pointed finite linear graph} is a finite linear graph with one special designated end vertex called the \emph{zero-vertex}. As suggested by the name, a convenient notation for working with finite linear graphs is the following: we label the vertices of a graph with $n$ vertices by $0,1,\ldots, n-1$ so that $0$ corresponds to the zero-vertex and vertices $i$ and $j$ are connected by an edge if and only if $\norm{i-j}\leq 1$. By $\fl{n}$, we denote the pointed finite linear graph with $n$ vertices labeled as just described. An \emph{epimorphism} between pointed finite linear graphs is a function $f:V \to W$ that is a graph homomorphism (i.e., $f$ takes edges to edges) which is surjective on vertices and edges and which preserves the zero-vertex. 

To put things into the model-theoretic context in which Fraissé theory is usually developed: we are working with a language $\mathcal{L}=\{R,c\}$ consisting of one binary relation symbol and one constant symbol respectively and with the class of finite $\mathcal{L}$-structures in which $R$ is interpreted as an edge relation satisfying (1)-(3) above and $c$ is interpreted as an endpoint. Epimorphisms as defined above correspond exactly to the usual model theory notion of epimorphisms between structures.

We will consider a Fraissé category in which objects are pointed finite linear graphs and morphisms are restricted to a proper subset of all epimorphisms. 

\begin{definition}\label{defn1}
Let $\mathcal{K}$ be the category of pointed finite linear graphs where morphisms are all epimorphisms, $f:\fl{n} \to \fl{m}$, such that: there exist $0= i_0<i_1<\cdots<i_k=n-1$ such that
\begin{enumerate}
    \item $f \vert_{[i_j,i_{j+1}]}$ is either non-increasing or non-decreasing for $0\leq j\leq k-1$
    \item $f(i_j)=m-1$ for $j$ odd 
    \item $f(i_j)=0$ for $j$ even
\end{enumerate}
\end{definition}

The morphisms are essentially discrete versions of the open, continuous surjections  $[0,1] \to [0,1]$. For a morphism $f$ as in Definition \ref{defn1}, define the \emph{degree} of $f$, denoted $\textrm{deg}(f)$, to be $k$.

Following \cite{pan-sol} and \cite{irw-sol}, we say that a class $\mathcal{F}$ of finite $\mathcal{L}$-structures with a fixed family of morphisms is a \emph{projective Fraissé family} if: 
\begin{enumerate}
    \item $\mathcal{F}$ is a category--morphisms are closed under composition and the identity map on each structure is a morphism
    \item $\mathcal{F}$ contains countably many structures (up to isomorphism)
    \item for any $A, B \in \mathcal{F}$, there exist $C \in \mathcal{F}$ and morphisms $C \to A$ and $C \to B$
    \item for any $A,B,C \in \mathcal{F}$ and morphisms $f:B \to A$ and $g:C \to A$, there exists $D \in \mathcal{F}$ and morphisms $f':D \to B$ and $g':D \to C$ so that $g \circ g'=f\circ f'$
\end{enumerate}
Properties (3) and (4) above are called the \emph{joint-projection property} and the \emph{projective amalgamation property} respectively. 

\begin{thm}\label{thm1}
 $\mathcal{K}$ is a projective Fraissé family.
\end{thm}

 Points (1)-(3) in the definition are easy to check: we will show that the projective amalgamation property holds. The proof proceeds by amalgamating ``slope-by-slope'' via a sequence of lemmas. The following notation is convenient. For $f:\fl{n}\to\fl{m}$ a function, we let $\str(f) \in [m]^n$ be the string $f(0),f(1),\ldots, f(n-1)$. For a finite string $a=a_0,a_1,\ldots ,a_{n-1}$, we denote by $\rev(a)$ the string $a_{n-1}, a_{n-2},\ldots ,a_0$. For $a=a_0,a_1,\ldots,a_{n-1}$ and $0\leq i \leq j \leq  n-1$, by $a \restriction_i^j$ we mean the string $a_i,a_{i+1},\ldots,a_j$. The symbol $a^\frown b$ denotes the concatenation of string $a$ with string $b$. 
 
 Let $\mathcal{K}_{\textrm{inc}}$ be the set of non-decreasing morphisms in $\mathcal{K}$. Let $\mathcal{D}$ be the set of non-increasing functions between objects of $\mathcal{K}$ which send edges to edges and are surjective on both vertices and edges. Note that the members of $\mathcal{D}$ are not technically epimorphisms since they do not preserve the zero-vertex.

\begin{lem}\label{lem1}
Let $f:\fl{n}\to\fl{m}$ be a morphism. For all $i \in [n]$, $j\in \N$, if we let 
\[s=(\str(f)\restriction_0^{i-1})^\frown \underbrace{f(i)f(i)\cdots f(i)^\frown}_{j-\textrm{times}} (\str(f)\restriction_{i+1}^{n-1})\]
then there is $\phi \in \mathcal{K}_{\textrm{inc}}$, $\phi:\fl{n+j}\to \fl{n}$ with $s=\str(f \circ \phi)$. 
\end{lem}

\begin{proof}
Let $\phi(l)=\begin{cases}
l & \textrm{ if }1\leq l\leq i-1\\
i & \textrm{ if }i\leq l\leq i+j-1\\
l-j+1 & \textrm{ if }i+j\leq l\leq n+j-1\\
\end{cases}$

One can check that $\str(f\circ \phi)= s$ and clearly $\phi$ is non-decreasing.
\end{proof}

Applying Lemma \ref{lem1} $n$ times and composing yields the following:

\begin{lem}\label{lem2}
Let $f:\fl{n}\to\fl{m}$ be a morphism, $i_0,i_1,\ldots,i_{n-1} \in \N$ and let
\[s=\underbrace{f(0) \cdots f(0)}_{i_0 \textrm{ times }}\underbrace{f(1) \cdots f(1)}_{i_1 \textrm{ times }}\cdots \underbrace{f(n-1) \cdots f(n-1)}_{i_{n-1} \textrm{ times }}\]  
Then there is $\phi \in \mathcal{K}_{\textrm{inc}}$ so that $\str(f\circ \phi)=s$. 
\end{lem}

We can now amalgamate along one slope:

\begin{lem}\label{lem3}
Let $f:\fl{m}\to \fl{k}$ and $g:\fl{n}\to \fl{k}$ both be in $\mathcal{K}_{\textrm{inc}}$ or in $\mathcal{D}$. Then there exist $l\in\N$ and morphisms $\tilde{f}:\fl{l}\to\fl{m}$ and $\tilde{g}:\fl{l}\to\fl{n}$ so that $f \circ \tilde{f}=g\circ \tilde{g}$. Further $\tilde{f}$ and $\tilde{g}$ are both in $\mathcal{K}_{\textrm{inc}}$.
\end{lem}

\begin{proof}
Suppose that $f,g \in \mathcal{K}_{\textrm{inc}}$. As $f$ and $g$ are non-decreasing, they have the following form:
\[\str(f)= \underbrace{0 0 \cdots 0}_{f_0 \textrm{ times }}\underbrace{1 1 \cdots 1}_{f_1 \textrm{ times }}\cdots \underbrace{(k-1) (k-1) \cdots (k-1)}_{f_{k-1} \textrm{ times }}\]
\[\str(g)= \underbrace{0 0 \cdots 0}_{g_0 \textrm{ times }}\underbrace{1 1 \cdots 1}_{g_1 \textrm{ times }}\cdots \underbrace{(k-1) (k-1) \cdots (k-1)}_{g_{k-1} \textrm{ times }}\]
We let $m_i=\max\{f_i,g_i\}$ for $0\leq i\leq k-1$ and let
\[m= \underbrace{0 0 \cdots 0}_{m_0 \textrm{ times }}\underbrace{1 1 \cdots 1}_{m_1 \textrm{ times }}\cdots \underbrace{(k-1) (k-1) \cdots (k-1)}_{m_{k-1} \textrm{ times }}\]
Then set $l=\norm{m}$ and apply Lemma \ref{lem2} to find $\tilde{f}$ and $\tilde{g}$ so that $\str(f\circ \tilde{f})=m$ and $\str(g \circ \tilde{g})=m$. 

A nearly identical argument proves the claim for $f,g \in \mathcal{D}$.
\end{proof}

An \emph{interval} in $\fl{m}$ is just a subset of the form $\{a \in \fl{m} \ :\ i\leq a \leq j\}$ for some $0\leq i\leq j\leq m-1$ and we denote such an interval by $[i,j]$. For an interval $I \subseteq \fl{n}$, let $\max(I)$ be the greatest integer such that $\max(I) \in I$ and let $\min(I)$ be the least integer such that $\min(I) \in I$. We write $I<J$ for intervals $I, J$ if $\max(I)<\min(J)$.

\begin{proof}[Proof of Theorem]\ref{thm1}
Let $f:\fl{m}\to\fl{k}$ and $g:\fl{n}\to\fl{k}$ be morphisms. We want to find $f', g' \in \mathcal{K}$ with common domain such that $f\circ f'=g\circ g'$.

We claim that it is enough to prove the Theorem for morphisms $f$ and $g$ satisfying the condition below:

A morphism $h:\fl{m} \to \fl{k}$ has condition $(*)$ if there exist $I_1<I_2<\cdots <I_{l_1}$, disjoint intervals of $\fl{m}$ so that $\bigcup_{j=1}^{l_1} I_j =\fl{m}$ and for each $j$, $h \restriction I_j:I_j \to \fl{k}$ is either a non-decreasing surjection or non-increasing surjection. Suppose that $f,g \in \mathcal{K}$. The fact that $f \in \mathcal{K}$ implies that there is $0=i_0<i_1<\cdots <i_k=m-1$ as in Definition \ref{defn1}. Then, consider the string 
\begin{align*}
s &=((\str(f))\restriction_0^{i_1-1})^\frown f(i_1)f(i_1)^\frown ((\str(f))\restriction_{i_1+1}^{i_2-1})^\frown \\ 
&\qquad ^\frown f(i_2)f(i_2)^\frown\cdots^\frown f(i_{k-1})f(i_{k-1})^\frown ((\str(f))\restriction_{i_{k-1}+1}^{m-1}) 
\end{align*}
and by Lemma \ref{lem2}, find $\phi$ so that $\str(f \circ \phi)=s$. It is easy to see that $f \circ \phi$ has condition $(*)$. We do the same with $g$ to produce $g \circ \psi$ with the property above. Then, notice that amalgamating over $f \circ \phi$ and $g \circ \psi$ produces an amalgamation over $f$ and $g$ as desired.

So from now on we assume $f$ and $g$ have $(*)$. Divide $\fl{m}$ into intervals $I_1<I_2< \ldots< I_{l_1}$ with $\bigcup_{j=1}^{l_1}I_j=\fl{m}$ so that for each $j$, $f\restriction_{I_j}:I_j \to \fl{k}$ is a non-decreasing surjection or a non-increasing surjection. Similarly subdivide $\fl{n}$ into intervals $J_1,\ldots , J_{l_2}$ for $g$.

Let $f_j:\fl{\norm{I_j}} \to \fl{k}$ be the function given by $f_j(i)= f(i+\min(I_j))$ and let $g_k:\fl{\norm{J_k}} \to \fl{k}$ be given by $g_k(i)= g(i+\min(J_k))$. Define $f_{-j}: \fl{\norm{I_j}} \to \fl{k}$ as $f_{-j}(i)= f(\max(I_j)-i)$ and $g_{-k}:\fl{\norm{J_k}} \to \fl{k}$ by $g_{-k}(i)= g(\max(J_k)-i)$. For any $j,k$, $f_j,f_{-j},g_k,g_{-k} \in \mathcal{K}_{\textrm{inc}}\cup \mathcal{D}$. When $j,k\in\Z$ are such that $f_j$ and $g_k$ are both in $\mathcal{K}_{\textrm{inc}}$ or both in $\mathcal{D}$, we let $p_{j,k} \in \N$ and let $f_{j,k}':\fl{p_{j,k}}\to \fl{\norm{I_j}}$ and $g_{j,k}':\fl{p_{j,k}}\to \fl{\norm{J_k}}$ be non-decreasing morphisms such that
\[f_j \circ f_{j,k}' = g_k \circ g_{j,k}'\]
The existence of $p_{j,k}, f_{j,k}'$, and $g_{j,k}'$ is by Lemma \ref{lem3}. 

We define two strings $\alpha$ and $\beta$ as follows. Let $a_1= 1,2,\ldots, l_1$ and $a_2=-l_1,-l_1+1, \ldots -1$ and let $b_1= 1,2,\ldots, l_2$ and $b_2=-l_2,-l_2+1, \ldots -1$. We refer to $a_1, a_2, b_1,$ and $b_2$ as \emph{blocks}. Then 
\[\alpha= a_1^\frown a_2^\frown a_1^\frown a_2 \cdots^\frown a_i\]
where there are $l_2$ blocks in the concatenation above and $i \in \{1,2\}$ depending on the parity of $l_2$. Also
\[\beta= b_1^\frown b_2^\frown b_1^\frown b_2 \cdots^\frown b_i\]
where there are $l_1$ blocks in the concatenation above. Note that $\norm{\alpha}=\norm{\beta}=l_1l_2$.

Further for all $1 \leq j\leq l_1l_2$, $f_{\alpha(j)}$ and $g_{\beta(j)}$ are both in $\mathcal{K}_{\textrm{inc}}$ or both in $\mathcal{D}$. This can be seen via two observations: (1) $f_{\alpha(1)}=f_1$ and $g_{\beta(1)}=g_1$ are both in $\mathcal{K}_{\textrm{inc}}$ and 
(2) for all $j$, and $i \in \{f,g\}$ $i_{\alpha(j)}\in \mathcal{K}_{\textrm{inc}} \implies i_{\alpha(j+1)}\in \mathcal{D}$ and $i_{\alpha(j)}\in \mathcal{D} \implies i_{\alpha(j+1)}\in \mathcal{K}_{\textrm{inc}}$. It is easy to see that (2) holds so long as $j$ and $j+1$ are in the same block. If $j+1$ is in a different block from $j$, then note that $\alpha(j+1)=-\alpha(j)$ and so (2) follows. 

We set 
\[L=\sum_{1\leq j\leq l_2l_2}p_{\alpha(j),\beta(j)}\]
and note that $\fl{L}$ is naturally formed as the disjoint union of $l_1l_2$ intervals of the form:
\[M_j=\left[\sum_{0\leq i\leq j-1}p_{\alpha(i),\beta(i)}, \sum_{1\leq i\leq j}p_{\alpha(i),\beta(i)}-1\right]=:[l_j,r_j]\]
where we set $p_{\alpha(0),\beta(0)}=0$. 

We now define $f':\fl{L}\to \fl{m}$ interval-by-interval as follows:

For $i \in M_1$, we let $f'(i)=f'_{\alpha(1),\beta(1)}(i)$. Suppose $f'$ is defined on $M_i$ for all $i \leq j-1$. If $\alpha(j)>0$, then for $i \in M_j$, we let $f'(i)= f'_{\alpha(j),\beta(j)}(i-l_j)+f'(r_{j-1})+1$. If $\alpha(j)<0$, then for $i \in M_j$, we let $f'(i)=f'(r_{j-1})-f'_{\alpha(j),\beta(j)}(i-l_j)-1$.

And similarly, we define $g':\fl{L}\to \fl{n}$ interval-by-interval: For $i \in M_1$, we let $g'(i)=g'_{\alpha(1),\beta(1)}(i)$. Suppose $g'$ is defined on $M_i$ for all $i \leq j-1$. If $\alpha(j)>0$, then for $i \in M_j$, we let $g'(i)= g'_{\alpha(j),\beta(j)}(i-l_j)+g'(r_{j-1})+1$. If $\alpha(j)<0$, then for $i \in M_j$, we let $g'(i)=g'(r_{j-1})-g'_{\alpha(j),\beta(j)}(i-l_j)-1$.

This choice of $f'$ and $g'$ work; one can verify that $f',g'\in \mathcal{K}$ and $f\circ f'=g\circ g'$.
\end{proof}

\section{The Fraissé limit of $\mathcal{K}$}\label{sec4fraisselimit}

\subsection{Enlarging category $\mathcal{K}$}
We first must define an enlargement of the category $\mathcal{K}$ to include possible limit objects. This type of enlargement is used in \cite{pan-sol} and we follow the development from that paper. A \emph{topological pointed graph},  $(X, R^X, c^X)$, consists of a compact, metrizable, zero-dimensional domain space, $X$; a reflexsive, and symmetric relation $R^X \subseteq X^2$ which is a closed subset of $X^2$; and a designated point $c^X \in X$. Epimorphisms between topological pointed graphs must take edges to edges, preserve the designated point, be surjective on vertives and edges, and moreover must be continuous. Of course, finite pointed linear graphs (with the discrete topology) and epimorphisms as defined earlier for finite pointed linear graphs are examples of topological pointed graphs and epimorphisms respectively.

We will be interested only in a proper subset of topological pointed graphs and of morphisms between them. Loosely, we want those structures which are ``approximated" by the category $\mathcal{K}$ and therefore are possible limit objects for $\mathcal{K}$ and subcategories of $\mathcal{K}$. Define the category $\mathcal{K}^\omega$ as follows. If $(K_n, f_n^{n+1})$ is a sequence consisting of finite pointed linear graphs, $K_n$, and where each $f_n^{n+1}:K_{n+1} \to K_n$ is in $\mathcal{K}$, then there is a natural way in which to view $\mathbb{K}:=\varprojlim (K_n, f_n^{n+1})$ as a topological pointed graph. The topology on $\bK$ is the compact, zero-dimensional product topology that arises by taking each $K_n$ as a finite discrete space. The graph relation, $R^{\mathbb{K}}$, defined by 
\[((x_n)_{n\in\N}, (y_n)_{n\in\N}) \in R^{\mathbb{K}} \iff \forall n\in \N \left((x_n,y_n) \in R^{K_n}\right) \] 
is easily seen to be reflexsive and symmetric and the designated point is given by $c^{\mathbb{K}}=(c^{K_n})_{n\in\N}$.
Now, $\mathcal{K}^\omega$ contains all topological pointed graphs of this kind: those which arise as inverse limits of sequences of finite pointed linear graphs with bonding maps from $\mathcal{K}$. In particular, $\mathcal{K}$ is contained in $\mathcal{K}^\omega$. Morphisms in $\mathcal{K}^\omega$ are defined as follows: given $\mathbb{K}=\varprojlim (K_n,f_{n}^{n+1})$ and $\mathbb{L}=\varprojlim (L_n,g_n^{n+1})$ in $\K^\omega$, $\phi:\mathbb{K}\to\mathbb{L}$ is a morphism if there is an increasing sequence $i_1<i_2<i_3<\cdots$ of natural numbers and maps $\phi_n:K_{i_n} \to L_n$ with each $\phi_n \in \K$ and so that for each $n$, $\pi^{\bL}_n\circ \phi=\phi_n\circ \pi^{\bK}_{i_n}$ where $\pi^{\bL}_m:\mathbb{L} \to L_m$ is the projection of $\mathbb{L}$ onto its $m$th coordinate and $\pi^{\bK}_m$ is defined analogously for $\mathbb{K}$. It can be easily seen that morphisms $\mathbb{L} \to \mathbb{K}$ in $\K^\omega$ are a proper subset of all epimorphisms between $\mathbb{K} \to \mathbb{L}$ as topological pointed graphs. Notice that if $\mathbb{L}=\varprojlim (L_n, g_n^{n+1}) \in \K^\omega$ and $A \in \K$, then morphisms $\mathbb{L} \to A$ are all maps of the form $f\circ \pi^{\bL}_n$ for some $n\in\N$ and some $f:L_n \to A$ in $\K$. For example, the canonical projections $\pi_n^{\bL}:\mathbb{L} \to L_n$ are morphisms. 

For $\mathbb{L} \in \K^\omega$, we call $\mathbb{L}$ a \emph{pre-continuum} if $R^{\mathbb{L}}$ is transitive (i.e., an equivalence relation). Each pre-continuum has a \emph{quotient}, that is, the space $\mathbb{L} / R^{\mathbb{L}}$ with the quotient topology. This space is a compact metrizable space since the relation $R^\mathbb{L}$ is always a closed relation and the topology on $\bL$ is always compact, metrizable. The terminology of pre-continuum is justified by the following:

\begin{thm}\label{thm2}
A topological space $X$ is a Knaster continuum if and only if $X$ is homeomorphic to the quotient of some pre-continuum in $\K^\omega$.
\end{thm}

For the proof, recall that the \emph{mesh} of an open cover $\mathcal{U}$ of a metric space is given by $\sup \{\textrm{diam}(U) \ : \ U \in \mathcal{U}\}$. We will also need the following folklore/ easy to check fact. If $\mathbb{I}$ is a topological pointed graph formed as the inverse limit of finite pointed linear graphs such that each bonding map is a monotone epimorphism (monotone meaning that pre-images of connected sets are connected) and $R^{\mathbb{I}}$ is an equivalence relation, then $\mathbb{I} / R^{\mathbb{I}}$ is homeomorphic to an arc. Further, if $q: \mathbb{I} \to \mathbb{I} /R^{\mathbb{I}}$ is the quotient map, then $q(c^{\mathbb{I}})$ is an endpoint of the arc $\mathbb{I} / R^{\mathbb{I}}$. An \emph{endpoint} of a continuum $X$ is a point $x \in X$ such that if $A$ is any arc in $X$ with $x \in A$, then $x$ is an endpoint of $A$. 

In the proof below we will at one point mention \emph{the property of Kelley} (a property of continua) but we will never actually use the definition so we do not define it here. For a definition, see \cite{char-roe}, Definition 6.2 on p.23. Following Charatonik and Roe in \cite{char-roe}, we say that an epimorphism between finite linear graphs, $f:B \to A$ is \emph{confluent} if for any connected subset $A' \subset A$, every connected component, $P$, of $f^{-1}(A')$ is such that $f(P)=A'$ (see \cite{char-roe}, Definition 4.1 on p.15).

\begin{proof}
Let $\mathbb{L}=\varprojlim (L_n, g_n^{n+1})$ be a pre-continuum in $\K^\omega$, let $\pi_n:\mathbb{L} \to L_n$ be the projection map on the $n$th coordinate, and let $L=\mathbb{L} /R^{\mathbb{L}}$. We want to show that $L$ is a Knaster continuum. As per a characterization due to Krupski (Theorem 3 of \cite{krupski}), $L$ is a Knaster continuum if:
\begin{enumerate}
    \item $L$ is a chainable continuum
    \item $L$ has the property of Kelley
    \item  every proper subcontinua of $L$ is an arc
    \item $L$ has one or two endpoints
\end{enumerate}

Point (1) above follows from applying Lemma 4.3 of \cite{irw-sol} since any open cover of $\bL$ is refined by some $\pi_n$, which is an epimorphism onto a finite linear graph.

The fact that $L$ has the property of Kelley is a consequence of Theorem 6.5 of \cite{char-roe} along with the observation that every morphism in $\mathcal{K}$ is confluent. Let $K$ be a proper subcontinua of $L$. Then, we have that for each $n$, $\pi_n(q^{-1}(K))$ is an $R^{L_n}$-connected subset of $L_n$. A moment of thought about the definition of the maps in $\mathcal{K}$ checks that $q^{-1}(K)$ must be an inverse limit of connected linear graphs with monotone bonding maps; so $K$ must be an arc. This proves (2) and (3).

To see (4), first we show that $q(c^\mathbb{L})$ is an end point. If $A$ is an arc in $L$ containing $q(c^{\bL})$, then for each $n$, $\pi_n(q^{-1}(A))$ is an $R^{L_n}$-connected subset of $L_n$ containing $c^{L_n}$, i.e., $\pi_n(q^{-1}(A))$ is a finite pointed linear graph. Since each $\pi_n(q^{-1}(A))$ is connected, for all but (possibly) finitely many $n \in \N$, the restriction of $g_n^{n+1}$ to $\pi_{n+1}(q^{-1}(A))$ is a monotone map from $\pi_{n+1}(q^{-1}(A))$ onto $\pi_n(q^{-1}(A))$ that sends $c^{L_{n+1}}$ to $c^{L_n}$. Now it follows that $q(c^{\bL})$ is an endpoint of $A$. If every map $g_n^{n+1}$ has odd degree, then a similar argument shows that the point $q(d)$ where $d(n)$ is the endpoint of $L_n$ not equal to $c^{L_n}$ is also an endpoint. If $x$ is any other point in $L$, then we can show that $x$ is not an endpoint as follows: let $x=q(x')$ for some $x' \in \bL$. There is some $n$ so that $\pi_n(x')$ is not an endpoint of $L_n$. Let $V_n$ be an $R^{L_n}$-connected subset of $L_n$ so that $\pi_n(x') \in V_n$ and each $R^{L_n}$-neighbor of $\pi_n(x') \in V_n$. Let $a$ and $b$ denote the two distinct $R^{L_n}$-neighbors of $\pi_n(x')$ other than $\pi_n(x')$ itself. For $m>n$, let $V_m$ be the connected component of $(g_n^m)^{-1}(V_n)$ containing $\pi_m(x')$. Then: $\overline{V}=\varprojlim (V_n,g_n^{n+1}\restriction V_{n+1})$ is an inverse limit of finite linear graphs with monotone bonding maps; so $q\left(\overline{V}\right)$ is an arc contained in $L$. Because $R^{\bL}$ is transitive, 
\[d_{L^m}((g_n^m)^{-1}(a)\cap V_m,(g_n^m)^{-1}(b)\cap V_m) \to \infty \textrm{ as }m \to \infty \]
and it follows that $x$ is an interior point of  $q\left(\overline{V}\right)$. By the theorem of Krupski, $L$ is a Knaster continuum.

Now, let $M=\varprojlim (I, T_n)$ be a Knaster continuum. By Lemma 4 of \cite{debski}, up to homeomorphism, we may assume that each $T_n:I \to I$ is a standard tent-map of the form:
\[T_n(x)=\begin{cases}
dx & \textrm{ if } x\in \left[\frac{m}{d},\frac{m+1}{d}\right] \textrm{ and $m$ is even}\\
1+m-dx & \textrm{ if } x\in \left[\frac{m}{d},\frac{m+1}{d}\right] \textrm{ and $m$ is odd}\\
\end{cases}\]
where $d$ is the degree of $T_n$. Denote by $\pi^M_n:M \to I$ the projection of $M$ onto it's $n$th coordinate. For each $n \in \N$, we construct a chain $\mathcal{U}^n=\{U^n_0,U^n_1,\ldots ,U^n_{p(n)-1}\}$ on $[0,1]$ so that
\begin{enumerate}
    \item $0\in U^n_0$
    \item $U^n_i \cap U^n_j \neq \emptyset \iff \norm{i-j}\leq 1$
    \item if $\norm{i-j}>1$ and $k,l$ are such that $T_n\left(\overline{U^{n+1}_k}\right) \subset U^n_i$ and \\ $T_n\left(\overline{U^{n+1}_l}\right) \subset U^n_j$, then $\norm{k-l}>2$
    \item for each $i<p(n)$, there exists $k<p(n+1)$ so that $T_n\left(\overline{{U^{n+1}_k}}\right)\subset U^n_i$
    \item $\textrm{mesh}(\mathcal{U}^n)<\frac{1}{n}$
\end{enumerate}

\noindent and so that there exists $\epsilon_n >0$ so that

\begin{enumerate}
\setcounter{enumi}{5}
    \item $d(U^n_i,U^n_j)>\epsilon_n$ if $\norm{i-j}>1$
    \item for each $i$, there is $x \in U^n_i$ so that $d\left(\{x\}, \bigcup_{j\neq i}U^n_j\right)>\epsilon_n$
    \item if $X \subset [0,1]$ with $\textrm{diam}(X)<\epsilon_n$, then there is some $i$ so that $X \subset U^n_i$
\end{enumerate}

A chain satisfying properties (6)-(8) above is called an \emph{$\epsilon_n$-fine chain}. Given any $\delta >0$, it is easy to construct a chain on $[0,1]$ that is $\epsilon$-fine for some $\epsilon$ and so that the diameter of each element of the chain is less than $\delta$. 

We now show how to construct a sequence of chains $\mathcal{U}^n$ satisfying (1)-(6). Let $\mathcal{U}^0$ be a $\epsilon_0$-fine chain on $[0,1]$, for some $\epsilon_0>0$, with elements of the chain labelled so that conditions (1) and (2) are satisfied. Now given chain $\mathcal{U}^n$ and $\epsilon_n >0$ satisfying (1)-(6), we construct chain $\mathcal{U}^{n+1}$ as follows: let $\epsilon_{n+1}$ and $\mathcal{U}^{n+1}$ be such that there is an $\epsilon_{n+1}$-fine chain, $\mathcal{U}^{n+1}$,  on $[0,1]$ so that each element of $\mathcal{U}^{n+1}$ has diameter less than 

\[\min \left\{\frac{1}{n+1}, \frac{\epsilon_n}{6\textrm{deg}(T_n)}\right\}\]
Label the elements of $\mathcal{U}^{n+1}$ so that conditions (1) and (2) are satisfied. Clearly, (5) is satisfied. Notice that for each $k<p(n+1)$:
\[\textrm{diam}\left(T_n\left(\overline{U^{n+1}_k}\right)\right)<\frac{\epsilon_n}{3}\]
(we are using, above, that $T_n$ is a standard tent-map and therefore can expand the length of an interval by at most a factor of $\deg (T_n)$). So, by (7) for $\mathcal{U}^n$, and the fact that $T_n$ is surjective, we must have for each $i<p(n)$, some $k<p(n+1)$ with $T_n\left(\overline{U^{n+1}_k}\right)\subset U^n_i$--i.e, condition (4) is satisfied for $\mathcal{U}^{n+1}$. Condition (6) for $\mathcal{U}^n$ implies that (3) holds for $\mathcal{U}^{n+1}$. This concludes the construction.

Let $p(n)$ be the number of elements of chain $\mathcal{U}^n$. For $n \in \N$, let $A_n$ be a pointed linear graph with $p(n)$ vertices. We think of $A_n$ as having its vertices labelled by $0,1, \ldots , p(n)-1$ so that $c^{A_n}$ is vertex zero. Define $g_n:A_{n+1} \to A_n$ by
\[g_n(i)=\min \left\{j : T_n\left(\overline{U^{n+1}_i}\right)\subset U^n_j\right\}\]

Such a $j$ exists since $\textrm{diam} \left(T_n\left(\overline{U^{n+1}_i}\right)\right)<\frac{\epsilon_n}{3}$ and by condition (8) for chain $\mathcal{U}^n$. Clearly, $g_n(c^{A_{n+1}})=c^{A_n}$. Condition (4) implies that $g_n$ is surjective on vertices and condition (5) implies that $g_n$ is surjective on edges. The fact that $\textrm{diam} \left(T_n\left(\overline{U^{n+1}_i}\right)\right)<\frac{\epsilon_n}{3}$ and conditions (2) and (7) imply that $g_n$ preserves edges. So, $g_n$ is an epimorphism. Lemma \ref{obs0} for $T_n$ gives that $g_n \in \K$. 

Let $\mathbb{M}=\varprojlim (A_n,g_n)$. First, we will show that condition (3) implies that for any $a \in \bM$ there is at most one element $b \in \bM$ so that $a \neq b$ and $(a,b) \in R^{\bM}$. Assume for contradiction that $a,b,c \in \bM$ are distinct so that $(a,b) \in R^{\bM}$ and $(a,c) \in R^{\bM}$. Then, there is some $n \in \N$ so that $a(n), b(n)$, and $c(n)$ are  pairwise disjoint. Since $L_n$ is a linear graph, it follows that $\norm{b(n)-c(n)}>1$. So, by condition (3), $\norm{b(n+1)-c(n+1)}>2$ which implies that either $(a(n+1), b(n+1)) \notin R^{L_n}$ or $(a(n+1),c(n+1)) \notin R^{L_n}$, a contradiction. Since any element of $\bM$ is $R^{\bM}$-connected to at most one element (other than itself) of $\bM$, we have that $R^{\bM}$ is transitive.  

We claim that $\mathbb{M}/R^{\mathbb{M}}$ is homeomorphic to $M$. Define $\phi:\mathbb{M} \to M$ as follows. Given $x \in \mathbb{M}$, let 
\[\phi(x)=\bigcap_{n\in\N} (\pi^{M}_n)^{-1}\left(\overline{U^n_{x(n)}}\right)\]

Note that the cover $\{ (\pi^M_n)^{-1}(U)\ : \ U \in \mathcal{U}^n\}$ has mesh at most $\frac{1}{2^n}$. 

The sequence $\left((\pi^{\bM}_n)^{-1}\left(\overline{U^n_{x(n)}}\right)\right)_{n\in\N}$ is a sequence of nested compact sets whose diameters converge to zero, so there is a unique point in the intersection above. We claim that:
\begin{enumerate}
    \item $x R^{\mathbb{M}}y \iff \phi(x)=\phi(y)$
    \item $\phi$ is continuous
    \item $\phi$ is onto
\end{enumerate}
and from points (1)-(3) and the fact that $\mathbb{M}/R^{\mathbb{M}}$ is compact and $M$ is Hausdorff it follows that $\phi$ descends to a homeomorphism $\mathbb{M}/R^{\mathbb{M}} \to M$. 

To see (1), assume first that $x$ and $y$ are such that $(x,y) \in R^{\bM}$. It follows that for each $n$, $U^n_{x(n)}$ and $U^n_{y(n)}$ intersect. Notice that the sequence
\[\left((\pi_n^{\bM})^{-1}\left(\overline{U^n_{x(n)} \cup U^n_{y(n)}}\right)\right)_{n\in\N}\]
is a sequence of nested compact sets whose diameters converge to zero and that $\phi(x)$ and $\phi(y)$ are both in the intersection of this sequence of compact sets by the definition of $\phi$. It follows that $\phi(x)=\phi(y)$. If $(x,y) \notin R^{\bM}$, then there is some $n$ so that $\norm{x(n)-y(n)} >1$ and so $\overline{U^n_{x(n)}} \cap \overline{U^n_{y(n)}} =\emptyset$ by (3). Then, $\phi(x) \in (\pi^{\bM}_n)^{-1}(\overline{U^n_{x(n)}})$ and $\phi(y) \in (\pi^{\bM}_n)^{-1}(\overline{U^n_{y(n)}})$ and $(\pi^{\bM}_n)^{-1}(\overline{U^n_{x(n)}}) \cap (\pi^{\bM}_n)^{-1}(\overline{U^n_{y(n)}})=\emptyset$ so $\phi(x) \neq \phi(y)$. Point (2) follows essentially from the definition of the map $\phi$. 

To see (3), let $z \in M$. Consider the tree of all sequences of the form \\ $(s_0,\ldots, s_{n-1})$ where for all $i\leq 0$: $s_i \in A_i$, $g_{i}^{i+1}(s_{i+1})=s_i$, and $z(i) \in \overline{U}^i_{s_i}$. The partial order on the tree is given $s \preceq t$ if $t$ extends $s$. Then, it is easy to see that this tree is finitely branching and countably infinite; so Koenig's lemma implies that there is an infinite branch of the form $s_0, s_1,\ldots $. Note that $s$ defined by $s(n)=s_n$ is an element of $\bM$ and $z(n) \in U^n_{s(n)}$ for each $n$ which implies that $\phi(s)=z$.
\end{proof}

\subsection{Universal Knaster continuum}

The idea of classical Fraissé theory is that given a suitably rich category of finite structures and embeddings (where suitably rich means satisfies an injective amalgamation property), there is a unique ultrahomogeneous limit object: a countable structure so that every member of the category has an arrow into the limit object and every arrow between finite subobjects of the limit extends to an automorphism of the full limit. In the dual setting, the limit satisfies analogous properties with the arrow directions ``swapped." We will follow \cite{irw-sol} and \cite{pan-sol} for the development of projective Fraissé limits. 

\begin{thm}[Theorem 3.1 of \cite{pan-sol}]\label{thm3}
If $\mathcal{F}$ is a projective Fraissé family of finite (pointed) graphs, then there exists a unique up to isomorphism \\ (pointed) topological graph $\mathbb{F} \in \mathcal{F}^\omega$ so that 
\begin{enumerate}
    \item for each $A \in \mathcal{F}$, there is a morphism in $\mathcal{F}^\omega$ from $\mathbb{F} \to A$
    \item for $A,B \in \mathcal{F}$ and morphisms $f:\mathbb{F} \to A$ and $g:B \to A$, \\ 
    $f,g\in \mathcal{F}^\omega$, there is a morphism $h:\mathbb{F} \to B$ so that $g\circ h =f$
\end{enumerate}
\end{thm}

First a remark about point (2) of Theorem \ref{thm3} above: this property is sometimes called the \emph{projective expansion property}. Because of the definition of morphisms in $\mathcal{F}^\omega$, point (2) has the following consequence which we refer to as \emph{projective ultrahomogeneity}: for any morphisms $f,g:\mathbb{F} \to A$, there is some $h \in \aut(\mathbb{F})$ so that $f=g\circ h$. The proof of this fact is a simple diagram chase.

The structure $\mathbb{F}$ in Theorem \ref{thm3} is referred to as the \emph{projective Fraisse limit} of the family $\mathcal{F}$. It will be helpful in Section \ref{sec7degreeone} to understand a bit about how this limit is constructed, so we will say a few words about it here. Given a Fraissé family $\mathcal{F}$, we will say that a sequence $A_n$ of objects in $\mathcal{F}$ with $f_n:A_{n+1}\to A_n$ morphisms from $\mathcal{F}$ is a \emph{generic sequence} if the following two conditions hold:
\begin{enumerate}
    \item for any $A \in \mathcal{F}$, there is $A_n$ and $g$ in $\mathcal{F}$ so that $g:A_n \to A$
    \item for any $A,B \in \mathcal{F}$ and $e:A_n \to A$ and $g:B \to A$, there is some $N >n$ and $h:A_N \to B$ so that $g \circ h=e \circ f_n^N$
\end{enumerate}

The inverse limit of any generic sequence is the Fraissé limit of $\mathcal{F}$ (this will be the only fact that we need in Section \ref{sec7degreeone}). The construction of a generic sequence in a Fraissé category uses the amalgamation property repeatedly to build a sequence that is ``fully saturated" with the respect to the morphism types of $\mathcal{F}$ (of which there are only countably many). The paper \cite{irw-sol} develops the construction in detail. In the case that the graph relation on the Fraissé limit, $\mathbb{F}$, is an equivalence relation then we call $\mathbb{F}/R^{\mathbb{F}}$ the \emph{canonical quotient of $\mathcal{K}$}.

From now on, we denote by $\mathbb{K}$ the projective Fraissé limit of the category $\K$ defined in Section \ref{sec3fraisecategory}. 

\begin{prop}\label{prop1}
The canonical quotient of $\mathcal{K}$ is the universal Knaster continuum.
\end{prop}

\begin{proof}
Write $\bK=\varprojlim (K_n,f_n^{n+1})$ for $K_n$ a generic sequence. First we need to show that $\mathcal{K}$ has a canonical quotient, i.e., that $R^{\bK}$ is transitive. An inspection of the proof of Theorem \ref{thm2} reveals that $R^{\bK}$ is transitive (in fact, has equivalence classes of size at most 2) so long as the following condition holds: for any $K_n$ and any $x,y \in K_n$ with $d_{R^{K_n}}(x,y) >1$ there is some $N>n$ so that any $z \in (f_n^{N})^{-1}(x)$ and $w \in (f_n^{N})^{-1}(y)$ have the property that $d_{R^{K_n}}(z,w)>2$. Given $x,y \in K_n$, let $g:A \to K_n$ be an increasing morphism such that every $z \in g^{-1}(x)$ and every $w \in g^{-1}(y)$ are so that $d_{R^{A}}(z,w)>2$ (this is easy to arrange simply by ensuring that $g^{-1}(b)$ is an interval of length 2 for some $b$ that is between $x$ and $y$ in $K_n$). Apply property (2) of a generic sequence to find some $A_N$ and $h:A_N \to A$ with $g \circ h=f_n^N$ and note that this $N$ has the property desired.

Let $L$ be any Knaster continuum; by Theorem \ref{thm2}, there is 
\[\mathbb{L}=\varprojlim(L_n,g_n^{n+1}) \in \K^\omega\]
so that $L$ is homeomorphic to $\mathbb{L}/R^{\mathbb{L}}$. By property (1) in Theorem \ref{thm3} of projective Fraissé limits, there is a morphism $\phi_1:\mathbb{K} \to L_1$. By property (2), there is a morphism $\phi_2:\mathbb{K} \to L_2$ so that $g_1^2\circ \phi_2=\phi_1$. Continuing on in this way, we produce for each $n \in \N$, $\phi_n:\mathbb{K} \to L_n$ so that $g_n^{n+1} \circ \phi_{n+1}=\phi_n$. Then, define $\phi:\mathbb{K} \to \mathbb{L}$ by $\pi^{\mathbb{L}}_n(\phi(x))=\phi_n(x)$ for each $x\in\mathbb{K}$ and notice that $\phi$ is a morphism. As each $\phi_n$ sends edges to edges, $\phi$ sends edges in $\mathbb{K}$ to edges in $\mathbb{L}$. So $\phi$ descends to a continuous map $\phi':\mathbb{K} /R^{\mathbb{K}} \to \mathbb{L}/R^{\mathbb{L}}$.
\end{proof}

With some extra care, one can show that the map constructed in Proposition \ref{prop1} is open as well.

\section{Density of $\aut(\bK)$ in $\Homeo(K)$}\label{sec5density}
Proposition \ref{prop1} pushes the universality of the projective Fraissé limit (point (1) in Theorem \ref{thm3}) to the quotient of the limit. Using approximation results of D\k{e}bski, we can push projective ultrahomogeneity to the quotient of the limit. This gives an approximate ultrahomogeneity property for the universal Knaster continuum and more important-- establishes that $\aut(\bK)$ is dense in $\Homeo(K)$. 

From now on, let $q:\bK \to K$ denote the quotient map. We define a map $\Phi:\aut(\bK) \to \Homeo(K)$. Any automorphism $f:\bK \to \bK$ gives rise to a homeomorphism, $\Phi(f)$, of $(K)$ defined by
\[\Phi(f)(q(x))=q(f(x))\]
and further, $\Phi$ is a continuous embedding of $\aut(\bK)$ into $\Homeo(K)$. When $X$ is a compact space, we always take $\Homeo(X)$ with the topology induced by the supremum metric. We take $\aut(\bK)$ with the topology induced by considering it as a subgroup of the group of homeomorphisms of the topological space $\bK$. 

\begin{thm}\label{thm4.5}
Let $f,g: K \to L$ be continuous open surjections onto a Knaster continuum $L$ and let $\epsilon >0$. Then, there exists $h \in \aut(\bK)$ so that 
\[\dsup (f \circ \Phi(h),g) <\epsilon.\] 
\end{thm}

\begin{proof}
Let $\pi_n^L:L \to I$ be the projection onto the $n$th coordinate. Similarly, $\pi^{(K)}_n:K \to I$ is the projection of $K$ onto its $n$th coordinate. Let $\bK=\varprojlim (K_n,l_n^{n+1})$, where each $K_n \in\mathcal{K}$, and denote by $\pi_n:\bK \to K_n$ the projection of $\bK$ onto its $n$th coordinate. Let $\left(\mathcal{U}^n\right)_{n\in \N}$ be a sequence of open covers of $I$ as in the proof of Theorem \ref{thm2} (satisfying conditions (1)-(8)) for $L$. Fix $n$ be large enough so that  $\left\{(\pi^L_n)^{-1}(U) \ : \ U \in \mathcal{U}^n\right\}$ refines the cover of $L$ by balls of radius $\epsilon$. For ease of notation we let $\mathcal{U}^n =\{U_1,U_2,\ldots,U_N\}$ be labelled as in Theorem \ref{thm2}. Let $\delta >0$ be so that the chain $U_1,U_2,\ldots, U_N$ is $\delta$-fine. Consider $\pi^L_n \circ f :K \to I$; this is a continuous, open map that sends $q(c^{\bK})$ to 0. By a theorem of D\k{e}bski (Corollary on p. 207 of \cite{debski}), there is a sequence $f_i:I \to I$ of continuous open maps mapping 0 to 0 so that $f_i \circ \pi^{(K)}_i$ converges uniformly (in the supremum metric) to $\pi^L_n \circ f$. Let $j >n$ and $f_j:I \to I$ and $g_j:I \to I$ be continuous open maps such that:

\[d_{\sup}\left(\pi_n^L \circ f, f_j \circ \pi_j^{(K)}\right)<\delta\]
and

\[d_{\sup}\left(\pi_n^L \circ g, g_j \circ \pi_j^{(K)}\right)<\delta\]

Let $W_0,W_1,\ldots, W_M$ be a chain on $I$ of mesh less than $\delta'$ where $\delta'$ is chosen, by the uniform continuity of $f_j$, so that 

\[\norm{x-y}<\delta' \implies \norm{f_j(x)-f_j(y)}<\delta\]

\noindent and where the chain is labeled so that $W_i \cap W_j$ if and only if $\norm{i-j} \leq 1$ and $q(c^{\bK}) \in W_0$. Now let $m>j$ so that the cover
\[\left\{\left(\pi_m\right)^{-1}(a) \ :\ a \in K_m\right\}\]
refines the open cover
\[\left\{q^{-1}\left(\left(\pi_j^{(K)}\right)^{-1}\left(W_i\right)\right) \ : \ 0 \leq i\leq M \right\}\]

Define $\tilde{f_j}: K_m \to \fl{n}$ by
\[\tilde{f_j}(a)=\min \left\{i \ : \ f_j \circ \pi_j^{(K)} \circ q \left((\pi_m)^{-1}(a)\right) \subseteq U_i\right\}\]

\begin{claim} \label{claim1}
The map $\tilde{f}_j$ is in $\mathcal{K}$.
\end{claim}

\begin{proof}[Proof of Claim \ref{claim1}]
Notice that 
\[(\pi_m)^{-1}(c^{K_m}) \subset q^{-1}\left(\left(\pi_j^{(K)}\right)^{-1}\left(W_1\right)\right)\]

and that 
\[f_j\left(q^{-1}\left(\left(\pi_j^{(K)}\right)^{-1}\left(W_1\right)\right)\right) \subseteq U_1\]
So, $\tilde{f}_j(c^{K_m})=0$. To see that $\tilde{f}_j$ sends edges to edges, suppose that $a,b \in K_m$ and $(a,b) \in R^{K_m}$; then there exists $a' \in \pi_m^{-1}(a)$ and $b' \in \pi_m^{-1}(b)$ so that $(a',b') \in R^{\bK}$. So $q(a')=q(b')$ which implies that $\pi_j^{(K)}(q(\pi_m^{-1}(a))) \cap \pi_j^{(K)}(q(\pi_m^{-1}(b))) \neq \emptyset$. So
\[\textrm{diam} \left(\pi_j^{(K)}(q(\pi_m^{-1}(a))) \cup \pi_j^{(K)}(q(\pi_m^{-1}(b)))\right) <2\delta'\]
which implies in turn that
\[\textrm{diam} \left(f_j \circ \pi_j^{(K)} \circ q (\pi_m^{-1}(a)) \cup f_j \circ \pi_j^{(K)} \circ q(\pi_m^{-1}(b))\right)<2\delta\]
and so $(\tilde{f}_j(a),\tilde{f}_j(b)) \in R^{\fl{n}}$. The fact that $\tilde{f}_j \in \mathcal{K}$ follows from Lemma \ref{obs0} applied to the continuous open map $f_j$.
\end{proof}

Analogously, we may find $p \in \N$ and morphism $\tilde{g}_j:K_p \to \fl{n}$ so that 

\[\tilde{g}_j(a) =\min \left\{i \ : \ g_j \circ \pi_j^{(K)} \circ q(\pi_p^{-1}(a)) \subset U_i\right\}\]

Projective ultrahomogeneity of $\bK$ (see the remarks following Theorem \ref{thm3}) implies that there exists $h \in \aut(\bK)$ such that $\tilde{f_j}\circ \pi_m \circ h=\tilde{g_j}\circ \pi_p$. Now it is easy to check that $d_{\textrm{sup}}(f\circ \Phi(h),g)<\epsilon$.
\end{proof}

The following corollary is most important to us as we are interested in studying the extent to which the dynamics of $\aut(\bK)$ controls the dynamics of $\Homeo(K)$. 
\begin{cor}\label{cor1}
The image of $\Phi$ is dense in $\Homeo (K)$.
\end{cor}

\section{The notion of degree}\label{sec6degree}

In this section, we discuss an important definition, due to D\k{e}bski, of \emph{degree} of homeomorphisms of Knaster continua. This will be necessary later on when we compute the universal minimal flow of $\Homeo(K)$. As $K$ is a Knaster continuum, we can write $K=\varprojlim (I,t_n^{n+1})$ where each $I_n=I$ and each $t_n^{n+1}$ is a standard tent-map. Let $f:K \to K$ be a homeomorphism. D\k{e}bski proves (\cite{debski}, Corollary on p.207) that there is a sequence of continuous, open maps $f_i: I_i \to I$ so that 
\[f_i \circ \pi_i^{K} \rightarrow \pi_1^{K} \circ f\]
where the convergence above is with respect to the supremum metric on the space $C(K,I)$ of continuous maps $K \to I$. Further, the sequence
\[\frac{\textrm{deg}(f_i)}{\textrm{deg}(t_1^i)}\]
is eventually constant (Lemma 7 of \cite{debski}) and in fact independent of the  sequence $f_i$ chosen to approximate $\pi_1^K \circ f$. We let
\[\textrm{deg}(f):=\lim_{i\to \infty} \frac{\textrm{deg}(f_i)}{\textrm{deg}(t_1^i)} \]
and note that $\textrm{deg}(f)$ is a positive rational number.

Lemmas 11 and 12 of \cite{debski} tell us that for $f,g \in \Homeo(K)$, $\deg (f \circ g)=\deg (f)\deg(g)$ and that $\deg(\textrm{id}_{(K)})=1$. In particular, 
\[\deg:\Homeo(K) \to \Q^\times\]
is a group homomorphism, where $\Q^\times$ is the group of positive rationals with multiplication. Note that $\Q^\times$ is a free abelian group with countably many generators (the generators are the primes). Lemma 9 of the same paper implies that in fact $\deg$ is continuous when $\Homeo(K)$ is taken with the topology induced by the supremum metric and $\Q^\times$ with the discrete topology.

%\begin{lem}
%The map $\deg$ is surjective.
%\end{lem}

%\begin{proof}
%As $(K)$ is a universal Knaster continuum, for any prime $p$ and any $n \in \N$, there exists $N>n$ so that $p\vert \deg (t_N^{N+1})$ (Theorem 2.2 of \cite{eberhartlattice}). Then, the proof of Corollary 3.12 of \cite{eberhart} implies that every induced open map $(K) \to (K)$ has an inverse (i.e., is a homeomorphism since $(K)$ is compact and Hausdorff). So, by point (4) of Theorem 3.13 of the same paper, the image of the degree map applied to $\Homeo(K)$ is all of $\Q^\times$ aince there are induced open maps $(K) \to (K)$ of each possible degree.
%\end{proof}

We will at this point define an analogous definition of degree for automorphisms of $\bK=\varprojlim (K_n,f_n^{n+1})$. Suppose that $g:\bK \to \bK$ is an automorphism. By definition of morphisms in $\mathcal{K}^\omega$, there exists $i_1>0$ and $g_1:A_{i_1} \to A_1$ with $g_1 \in \mathcal{K}$ such that $\pi_1^{\bK} \circ g=g_1\circ \pi_{i_1}^{\bK}$. Then we define:
\[\deg (g)=\frac{\deg (g_1)}{\deg (f_1^{i_1})}\]
Recall that $\Phi:\aut(\bK) \hookrightarrow \Homeo(K)$ is the map induced by the quotient.

\begin{lem}\label{lem4}
For any $g \in \aut(\bK)$, $\deg (\Phi (g))=\deg (g)$. 
\end{lem}

\begin{proof}
Let $\bK=\varprojlim(K_n,f_n^{n+1})$ and $\pi_n$ the projection of $\bK$ onto its $n$th coordinate. For each $n \in \N$, let $p(n)$ be the number of vertices in $K_n$. Now, for $n \in \N$, we define a map $h_n: I \to I$ as follows. First, we define a map $u:\left[\deg (f_n^{n+1})(p(n)-1)+1\right] \to [p(n+1)]$. Let $u(0)=0$. For $1 \leq i \leq \deg(f_n^{n+1})(p(n)-1)$, let

\[
\begin{split}
u(i) &=\min \bigg\{ j \ : \ f_n^{n+1}(j) \neq f_n^{n+1}(j-1) \textrm{ and }\\
& \qquad \qquad \norm{\{k<j \ : \ f_n^{n+1}(k) \neq f_n^{n+1}(k-1)\}}=i-1 \bigg\}\\
\end{split}
\]

Then, we let $h_n^{n+1}:I \to I$ be the piece-wise linear map with $h_n^{n+1}(0)=0$ and with breakpoints at the points 
\[\left(\frac{u(i)}{p(n+1)-1},\frac{f_n^{n+1}(u(i))}{p(n)-1}\right)\]
for each $1 \leq i \leq \deg(f_n^{n+1})(p(n)-1)$. One can check that $h_n^{n+1}$ is a continuous, open surjection $I \to I$ with $\deg (h_n^{n+1})=\deg (f_n^{n+1})$. Given $\mathbf{x}=\left(x_j\right)_{j\in\N} \in \bK$ and some $l \in \N$, we define 
\[I_l(\mathbf{x})=\left[\frac{a_l(\mathbf{x})}{p(l)-1}, \frac{b_l(\mathbf{x})}{p(l)-1}\right]\]
where 
\[a_l(\mathbf{x})=\max \left(\left\{i  \ : \ \frac{i}{p(n)-1}< x_l\right\}\cup \{0\}\right)\]
and
\[b_l(\mathbf{x})=\min \left( \left\{i \ : \ \frac{i}{p(n)-1} > x\right\}\cup \{p(n)-1\}\right)\]

We have that the map
\[\bK \to \varprojlim (I,h_n^{n+1})\]
that sends a point $\mathbf{x}=\left(x_j\right)_{j\in\N} \in \bK$ to $\mathbf{y}=\left(y_j\right)_{j\in\N} \in \varprojlim (I,h_n^{n+1})$ where
\[y_j= \bigcap_{l\geq j} h_j^l\left(I_l(\mathbf{x})\right)  \]
descends to a homeomorphism 
\[\bK/R^{\bK} \to \varprojlim (I,h_n^{n+1})\]
There some details to check here, but the proofs are very similar to the proof of Theorem \ref{thm2}, so we omit the details. The point of this representation of $K$ is that it is clear that if $g \in \aut (\bK)$ such that $\pi_1 \circ g= g_1 \circ \pi_n$ for some $n \in \N$ and $g_1 \in \mathcal{K}$, then $\Phi(g)$ is such that:
\[\pi_1^{(K)} \circ \Phi(g) = g_1' \circ \pi_n^{(K)}\]
for some continuous open $g_1'$ of degree equal to the degree of $g_1$. So we have that
\[\deg (\Phi(g))=\frac{\deg g_1'}{\deg h_1^n}=\frac{\deg g_1}{\deg f_1^n}=\deg (g)\]
\end{proof}

\begin{lem}\label{lem5}
The degree map $\deg :\aut(\bK) \to \Q^\times$ is surjective.
\end{lem}

\begin{proof}
Let $\bK=\varprojlim(K_n,f_n^{n+1})$. Suppose that $\frac{p}{q} \in \Q^\times$ where $p$ and $q$ are positive natural numbers. Let $C \in \mathcal{K}$ be a finite pointed linear graph with enough vertices such that there exists morphisms of degree $p$ and morphisms of degree $q$ from $C$ to $K_1$. Fix $h:C \to K_1$ any morphism with $\deg (h)=q$. By the projective extension property, there is some $N \in \N$ and a morphism $t:K_N \to C$ so that $h \circ t=f_1^N$. So:
\begin{equation}\label{eqn4}
    \deg (t)=\frac{\deg (f_1^N)}{\deg (h)}=\frac{\deg (f_1^N)}{q}
\end{equation}
Now let $g:C \to A$ be a morphism with $\deg (g)=p$. Projective ultrahomogeneity implies that there is some $\alpha \in \aut(\bK)$ so that $\pi_1 \circ \alpha=g \circ t \circ \pi_N$ and now by \ref{eqn4}:
\[\deg (\alpha)= \frac{\deg (g) \deg(t)}{\deg (f_1^N)}=\frac{p}{q}\]
\end{proof}

Lemmas \ref{lem4} and \ref{lem5} imply that the degree map,
\[\deg:\Homeo(K) \to \Q^\times\]
is also surjective.

\begin{figure}[ht]
\centering
\includegraphics[height=4cm]{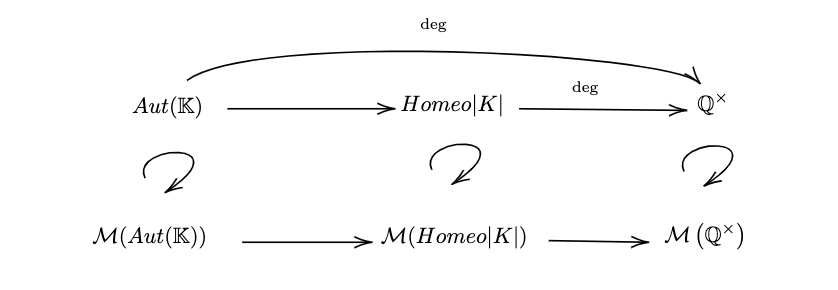}
\caption{}
\label{fig1}
\end{figure}

The map $\deg:\aut(\bK) \to \Q^\times$ is continuous as the composition of two continuous maps (see Figure \ref{fig1} for a diagram of the situation we have thus far). It is a general fact that if $G$ and $H$ are topological groups and there is a continuous surjective group homomorphism $f:G \to H$, then $f$ induces a surjective, continuous, $G$-equivariant map $\mathcal{M}(G) \to \mathcal{M}(H)$. This is simply because 
\[g \cdot x = f(g) \cdot x\]
defines a continuous action of $G$ on $\mathcal{M}(H)$ and the action is minimal because $f$ is surjective. This fact gives us the existence of the two maps on the bottom row of Figure \ref{fig1} and that they are both surjections. Notice that since $\mathcal{M}(\Q^\times)$ is non-metrizable (this is true of any countable discrete group by a theorem of Veech, see \cite{kpt}, Appendix), this immediately implies that the universal minimal flows of $\aut(\bK)$ and $\Homeo(K)$ are non-metrizable.

In the remainder of the paper, we will show that the bottom two arrows are actually injective as well; and hence that $\mathcal{M}(\aut(\bK))$ and $\mathcal{M}(\Homeo(K))$ are both homeomorphic to $\mathcal{M}(\Q^\times)$.

\section{The main theorem}\label{sec7degreeone}
In this section we prove the main theorem:

\begin{thm}\label{thm5.5}
The group $\Homeo(K) \simeq U \rtimes F$ where $U$ is Polish and extremely amenable and $F$ is free abelian on countably many generators.  
\end{thm}

Let: 
\[\Homeo^1(K)\leq \Homeo(K)\]
be the subgroup of all homeomorphisms with degree 1; that is the kernel of the degree map. By the considerations of the previous section, $\Homeo^1(K)$ is an open (and so Polish), normal subgroup of $\Homeo(K)$. This will end up being the group $U$ in Theorem \ref{thm5.5}. We defer the proof of Theorem \ref{thm5.5} for a bit as it requires some set-up and construction. We first will focus on showing that $\Homeo^1(K)$ is extremely amenable using a projective Fraissé construction and the Kechris-Pestov-Todorcevic correspondence. 

\subsection{Approximating the group of degree one homeomorphisms}
We define a Fraissé family, $\mathcal{K}^*$, that is a slight modification of $\mathcal{K}$ as follows. The objects are all pairs $(A,n)$ where $A$ is a finite pointed linear graph and $n\in\Q^\times$. For $(A,n)$ and $(B,m)$, $f:(B,m) \to (A,n)$ is a morphism if $f:B \to A$ is in $\mathcal{K}$ and $\deg (f)=\frac{m}{n}$. So, between any two fixed objects of $\mathcal{K}^*$, all morphisms have the same degree. In the case that $\frac{m}{n} \notin \N$, there are no morphisms between $(B,m)$ and $(A,n)$. Checking that $\mathcal{K}^*$ is a Fraissé family is easy. For example to see that the projective amalgamation property holds, given $f:(B,m) \to (A,k)$ and $g:(C,n) \to (A,k)$, let $D$ be the finite pointed linear graph and $f':D \to B$ and $g':D \to C$ be the morphisms constructed in the proof of Theorem \ref{thm1}. Then, an inspection of the proof of Theorem \ref{thm1} gives that $f'$ is a morphism $\left(D,\frac{mn}{k}\right) \to (B,m)$ and $g'$ is a morphism $\left(D,\frac{mn}{k}\right) \to (C,n)$; so $\left(D,\frac{mn}{k}\right), f',g'$ witness the amalgamation property for $\mathcal{K}^*$. From now on, $\bK^*$ will be used to denote the projective Fraissé limit of $\mathcal{K}^*$. There is an obvious forgetful functor from $\mathcal{K}^*$ to $\mathcal{K}$ which simply takes an object $(A,n)$ to $A$ and a morphism $f:(B,n) \to (A,m)$ just to the morphism of the underlying finite linear graph $B \to A$. Abusing notation slightly, if $f:(B,n) \to (A,m)$ is in $\mathcal{K}^*$, we will also use $f$ to denote the morphism in $\mathcal{K}$ of the underlying finite graphs $B \to A$.

\begin{prop}\label{prop2}
Any generic sequence for $\mathcal{K}^*$ is a generic sequence for $\mathcal{K}$. 
\end{prop}

\begin{proof}
Let $(A_n,m_n)$ and $f_n^{n+1}:(A_{n+1},m_{n+1}) \to (A_n,m_n)$ be a generic sequence in $\mathcal{K}^*$. Now, let $B \in \mathcal{K}$. Since $(B,1) \in \mathcal{K}^*$, we have a morphism $h:(A_N,m_N) \to (B,1)$ for some $N \in \N$ and of course, $h:A_N \to B$ is in $\mathcal{K}$. Suppose $C,B$ are finite pointed linear graphs and $h,g$ are morphisms in $\mathcal{K}$ with $h:A_n \to B$ and $g:C \to B$. Note that $\frac{m_n}{\deg (h)} \in \Q^\times$ so $\left(B, \frac{m_n}{\deg (h)}\right)\in \mathcal{K}^*$ and also $\left(C, \frac{m_n \deg (g)}{\deg (h)}\right)\in\mathcal{K}^*$. Further, $h:(A_n,m_n) \to \left(B, \frac{m_n}{\deg (h)}\right)$ and $g:\left(C, \frac{m_n \deg (g)}{\deg (h)}\right)\to \left(B, \frac{m_n}{\deg (h)}\right)$ are morphisms in $\mathcal{K}^*$. As $(A_n,m_n)$ is generic for $\mathcal{K}^*$, there is $N\geq n$ and $p:(A_N,m_N) \to \left(C, \frac{m_n \deg (g)}{\deg (f)}\right)$ so that $h \circ f_n^N=g \circ p$, and of course, $p \in \mathcal{K}$. So, the sequence $(A_n,f_n^{n+1})$ is generic for $\mathcal{K}$.
\end{proof}

So, let $\bK^*=\varprojlim ((A_n,m_n),f_n^{n+1})$ for a generic sequence $(A_n,m_n),f_n^{n+1}$ in $\mathcal{K}^*$. By Proposition \ref{prop2}, $\varprojlim (A_n,f_n^{n+1})=\bK$, the projective Fraissé limit of $\mathcal{K}$. Notice that any automorphism of $\bK^*$ is an automorphism of $\bK$. We have the following situation:
\[\aut(\bK^*) \overset{\textrm{id}}{\hookrightarrow} \aut(\bK) \overset{\Phi}{\hookrightarrow} \Homeo(K)\]

\begin{prop}\label{prop3}
The image, $\Phi(\aut (\bK^*))$, is dense in $\Homeo^1(K)$.
\end{prop}

\begin{proof}[Proof of Proposition \ref{prop3}]
We show that: 
\[\aut(\bK^*)=\{g \in \aut(\bK) \ : \ \deg(g)=1\}\]
and from this fact and Corollary \ref{cor1}, the proposition follows. First if $g \in \aut(\bK^*)$ and $g_1:(A_{i_1},m_{i_1}) \to (A_1,m_1)$ is in $\mathcal{K}^*$ so that $\pi_1 \circ g=g_1\circ \pi_{i_1}$, then note $\deg g_1= \frac{m_{i_1}}{m_1}$. But, $f_1^{i_1}:(A_{i_1},m_{i_1}) \to (A_1,m_1)$ is also a morphism in $\mathcal{K^*}$ and so $\deg (f_1^{i_1})=\frac{m_{i_1}}{m_1}$. It follows that $\deg(g)=1$. 

Conversely, if $g \in \aut(\bK)$ is a degree one automorphism, then let $i_1<i_2<\cdots$ and $g_j:A_{i_j} \to A_j$ so that $\pi_j \circ g=g_j \circ \pi_j$. We know that 
\[1=\deg (g)=\frac{\deg g_1}{\deg f_1^{i_1}}\]
and so $\deg(g_1)=\deg (f_1^{i_1})=\frac{m_{i_1}}{m_1}$. So $g_1:(A_{i_1},m_{i_1}) \to (A_1,m_1)$ is in $\mathcal{K}^*$. Now for any $j>1$, the fact that $f_1^j \circ g_j=g_1 \circ f_{i_1}^{i_j}$ implies
\[\deg (g_j)=\frac{\deg g_1 \deg f_{i_1}^{i_j}}{\deg f_1^j}=\frac{\frac{m_{i_1}}{m_1}\cdot \frac{m_{i_j}}{m_{i_1}}}{\frac{m_j}{m_1}}=\frac{m_{i_j}}{m_j}\]
and so $g_j:(A_{i_j},m_{i_j}) \to (A_j,m_j)$ is in $\mathcal{K}^*$. It follows that $g \in \aut(\bK^*)$.
\end{proof}

\subsection{Extreme amenability}
We prove in this section:
\begin{thm}\label{thm6}

\begin{enumerate}[(i)]
    \item The group $\aut(\bK^*)$ is extremely amenable.
    \item The group $Homeo^1(K)$ is extremely amenable.
\end{enumerate}

\end{thm}

Theorem \ref{thm6} will be proved via the dual of the Kechris-Pestov-Todorcecvic correspondence. First we need some set-up and definitions. By a \emph{$d$-coloring} of a set $X$, we simply mean a function $c:X \to [d]=\{0,1,\ldots,d-1\}$. Having fixed a coloring, $c:X \to [d]$, a subset $A$ of $X$ is \emph{monochromatic} if there is some $i \in [d]$ so that $A \subseteq c^{-1}(i)$. Fix a category $\mathcal{F}$ of finite structures with surjective morphisms. For $A,B$ objects in $\mathcal{F}$, let $\Epi(B,A)$ be the collection of all morphisms $B \to A$ in $\mathcal{F}$. Now we say that an object $A$ has \emph{the Ramsey property} if for any $d\in \N$ and any $B \in \mathcal{F}$, there exists $C \in \mathcal{F}$ so that for any coloring $c:\Epi(C,A) \to [d]$, there is some $g \in \Epi(C,B)$ so that 
\[\Epi(B,A) \circ g := \{f\circ g \ : \ f \in \Epi(B,A)\}\]
is monochromatic. The category $\mathcal{F}$ is a \emph{Ramsey category} if every object in $\mathcal{F}$ has the Ramsey property. Here is the dual of the Kechris-Pestov-Todorcecvic correspondence; it is due to Bartošová and Kwiatkowska (\cite{bkgowers}). An object $A \in \mathcal{F}$ is \emph{rigid} if there are no nontrivial morphisms $A \to A$.

\begin{thm}[Theorem 2.2. of \cite{bkgowers}]\label{thm4.75}
Let $\mathcal{F}$ be a projective Fraissé class \\ with projective Fraissé limit $\mathbb{F}$. The following are equivalent:
\begin{enumerate}
    \item the group $\aut(\mathbb{F})$ is extremely amenable
    \item $\mathcal{F}$ is a Ramsey category and every object in $\mathcal{K}$ is rigid
\end{enumerate}
\end{thm}

The Ramsey statement involved in the category $\mathcal{K}^*$, when reformulated appropriately, ends up being a direct application of the classical Ramsey theorem.  For $k\leq m$ natural numbers, we denote by $\textrm{II}([k],[m])$ the set of all increasing injections $[k] \hookrightarrow [m]$. If $f:[m] \to [n]$ is an increasing injection, by $f \circ \textrm{II}([k],[m])$ we mean the collection of all functions of the form $f \circ g$ where $g:[k]\to [m]$ is an increasing injection; of course every such $f \circ g$ is in $\textrm{II}([k],[n])$.

\begin{thm}[Finite Ramsey Theorem]\label{thm7}
Let $k\leq m$ and $d$ be natural numbers. Then, there exists $n \in \N$, $n>m$ so that for any coloring 
$c:\textrm{II}([k],[n]) \to [d]$, there exists $g \in \textrm{II}([m],[n])$ so that $g \circ \textrm{II}([k],[m])$ is monochromatic. 
\end{thm}

We denote by $R(k,m;d)$ the least such $n$ satisfying Theorem \ref{thm7} for $k,m,d$ as above.

\begin{proof}[Proof of Theorem \ref{thm6}]
Since by Proposition \ref{prop3} $\aut (\bK^*)$ is a dense subgroup of $\Homeo^1(K)$, point (i) of the theorem implies point (ii).  

By Theorem \ref{thm4.75}, to show that $\aut(\bK^*)$ is extremely amenable, we need to prove that $\mathcal{K}^*$ is a Ramsey class. So let $A=(\fl{k},p), B=(\fl{m},q) \in \mathcal{K}^*$ and let $d \in \N$. The statement of the Ramsey property for $A$ and $B$ is vacuously true if $\Epi(B,A)$ is empty so we may assume that $\Epi(B,A) \neq \emptyset$ and thus $\frac{q}{p} \in \N$. Now by Theorem \ref{thm7}, let $n=R\left(\frac{q}{p}\cdot (k-1),m;d\right)$ and we claim that $C=(\fl{n},q)$ witness the Ramsey property for $A$ and $B$ in $\mathcal{K}^*$; i.e, for any $c:\Epi(C,A) \to [d]$, there is some $g \in \Epi(C,B)$ so that $\Epi(B,A) \circ g$ is monochromatic.

Given $f \in \Epi(C,A)$ we associate to $f$ the increasing injection $\overline{f}:\left[\frac{q}{p}\cdot (k-1)\right] \to [n]$ defined as follows: 
\[\overline{f}(i)=\min \bigg\{j \in [n] \ : f(j) \neq f(j-1) \textrm{ and } \norm{\{l<j \ : \ f(l) \neq f(l-1)\}}=i \bigg\}\]

The key observation is that for a morphism $f: (C,q) \to (A,p)$ of degree $\frac{q}{p}$, there are exactly $\frac{q}{p}(k-1)$ distinct vertices, $v_1,v_2,\ldots, v_{\frac{q}{p}(k-1)}$ in $C$ so that $f(v_i) \neq f(v_i-1)$ and further the values of these $\frac{q}{p} (k-1)$ distinct vertices fully determine the morphism $f$. This implies that the set displayed above is nonempty for each $i<\frac{q}{p}(k-1)$ and thus that $\overline{f}(i)$ is defined for each $i \in \left[\frac{q}{p}(k-1)\right]$. It is clear from the definition of $\overline{f}$ that $\overline{f}(i+1)>\overline{f}(i)$ for each $i$ and so $f \in II\left(\left[\frac{q}{p}\cdot (k-1)\right] ,[n]\right)$. The observation at the beginning of the paragraph implies that map $f \to \overline{f}$ is injective.

Let $c:\Epi(C,A) \to [d]$ be a coloring. Define  \[c':\textrm{II}\left(\left[\frac{q}{p}\cdot (k-1)\right],[n]\right)\to [d]\]
by $c'(\overline{f})=c(f)$ (and $c'$ is arbitrary for elements of $\textrm{II}\left(\left[\frac{q}{p}\cdot (k-1)+1\right],[n]\right)$ not of the form $\overline{f}$). Let $p\in \textrm{II}([m],[n])$ be so that 
\[p \circ \textrm{II}\left(\left[\frac{q}{p}\cdot (k-1)\right],[m]\right)\] 
is $c'$-monochromatic. 

Let $g \in \Epi(C,B)$ be the unique degree one morphism defined by the condition:
\[\min \bigg\{i \in [n] \ : g(i)=j\bigg\}=p(j) \textrm{ for }j=1,\ldots,m-1\]

Now we are done so long as for each $h \in \Epi(B,A)$, 
\begin{equation}\label{eqn1}
    \overline{h \circ g} \in p\circ \textrm{II}\left(\left[\frac{q}{p}\cdot (k-1)\right],[m]\right)
\end{equation}

Take $h \in \Epi(B,A)$, notice that checking \ref{eqn1} for $\overline{h\circ g}$ is equivalent to showing that $\overline{h \circ g}(i) \in \textrm{codom}(p)$ for each $i \in \left[\frac{q}{p}\cdot (k-1)\right]$. To see this, compute that for any $i \in \left[\frac{q}{p}\cdot (k-1)\right]$:

\[
\begin{split}
    \overline{h \circ g} (i) & = \min \bigg\{j \in [n] \ : h \circ g(j) \neq h\circ g(j-1) \textrm{ and }\\
    & \qquad \qquad \norm{\{l<j \ : \ h\circ g(l) \neq h\circ g(l-1)\}}=i \bigg\} \\
    &= \min \bigg\{j \in [n] \ : \ g(j)=\min \big\{k : h(k) \neq h(k-1) \textrm{ and }\\
    & \qquad \qquad \norm{l<k \ : \ h(l) \neq h(l-1)}=i\big\} \bigg\}\\ 
    &= p\bigg(\min \big\{k : h(k) \neq h(k-1) \textrm{ and } \norm{l<k \ : \ h(l) \neq h(l-1)}=i\big\}\bigg)\\
\end{split}
\]
where we know that $k$ exists so that $h(k) \neq h(k-1) \textrm{ and } \norm{l<k \ : \ h(l) \neq h(l-1)}=i$  because $\deg (h) =\frac{q}{p}$.
\end{proof}

\subsection{The proof of Theorem \ref{thm5.5}}

We can now prove the main theorem.

\begin{proof}[Proof of Theorem \ref{thm5.5}]
From the diagram in Figure \ref{fig1}, we have a short exact sequence of Polish groups (every arrow below is continuous):
\[1 \hookrightarrow \Homeo^1(K) \hookrightarrow \Homeo(K) \overset{\deg}{\twoheadrightarrow}\Q^\times \twoheadrightarrow 1\]
By  \cite{jahelzucker} (p.9), to show that $\Homeo(K) \simeq \Homeo^1(K) \rtimes \Q^\times$ as a topological group, we need only show that the short exact sequence above \emph{splits continuously}, i.e., that there is a continuous group homomorphism $\Q^\times \to \Homeo(K)$ which is a right inverse of the degree map. It suffices to show that for each prime $p$ there exists a homeomorphism $f_p$ of degree $p$ such that for any $p \neq q$, $f_p$ and $f_q$ commute. Given this, we can extend the mapping $p \mapsto f_p$ to a group homomorphism $\Q^\times \to \Homeo(K)$ by first mapping $\frac{1}{p}\mapsto f_p^{-1}$ for each reciprocal of prime and then extending using the fact that elements of $\Q^\times$ can be uniquely factored into primes and reciprocals of primes.

Let $K=\varprojlim (I_n, T_n^{n+1})$ be the universal Knaster continuum and for each $n$, $\pi_n:K \to I_n$ the projection map onto the $n$th coordinate. We use maps from what is called in \cite{eberhart} the \emph{semigroup of standard induced maps} (see second paragraph on p. 129 of \cite{eberhart}). For $d >0$, let $g_d:I \to I$ be the standard degree-$d$ tent-map given by
\[g_d(x)=
\begin{cases}
dx & \textrm{ if }x \in \left[\frac{m}{d}, \frac{m+1}{d}\right] \textrm{ and }m \textrm{ is even}\\
1+m-dx & \textrm{ if }x \in \left[\frac{m}{d}, \frac{m+1}{d}\right] \textrm{ and }m \textrm{ is odd}\\
\end{cases}\]
It is easy to check that for $c,d>0$ the maps $g_c$ and $g_d$ commute. For each prime $p$, choose $f_p$ to be the homeomorphism of $K$ such that
\[\pi_1 \circ f_p =g_p \circ \pi_1\]
There is a unique such homeomorphism by Lemma 3.3  and Theorem 3.10 of \cite{eberhart}. Now let $p,q$ be primes. Observe that
\[\pi_1 \circ f_q \circ f_p=g_q \circ \pi_1 \circ f_p=g_q \circ g_p \circ \pi_1 =g_p \circ g_q \circ \pi_1\]
and similarly
\[\pi_1 \circ f_p \circ f_q=g_p \circ g_q \circ \pi_1 \]
By uniqueness (Lemma 3.3 of \cite{eberhart}), $f_q \circ f_p=f_p \circ f_q$. So we get that $\Homeo(K) \simeq \Homeo^1(K) \rtimes \Q^\times$. As noted before $\Q^\times$ is free abelian and generated by the set of primes. By Theorem \ref{thm6}, $\Homeo^1(K)$ is extremely amenable.
\end{proof}

\subsection{Dynamical consequences}

We return to the diagram in Figure \ref{fig1}. The kernel of the map 
\[\deg :\Homeo(K) \to \Q^\times\]
is extremely amenable by Theorem \ref{thm6}. The following Proposition is a general and, surely, folklore fact which applies to our situation. 

\begin{prop}\label{prop4}
Let $f:G \to H$ be a continuous, surjective group homomorphism between Polish groups and suppose that $K :=\ker (f)$ is extremely amenable. Then, the continuous surjective map $\mathcal{M}(G) \to \mathcal{M}(H)$ induced by $f$ is a $H$-flow isomorphism.
\end{prop}

\begin{proof}
Consider the action of $G$ on $\mathcal{M}(G)$; there is some $x_0 \in \mathcal{M}(G)$ fixed by $K$. We claim there is a well-defined action of $G / K$ on $\mathcal{M}(G)$ given by:
\begin{equation}\label{eqn2}
(gK) \cdot x=g\cdot x
\end{equation}
This follows from the fact that 
\begin{equation}\label{eqn3}
g_1K=g_2K \implies g_1 \cdot x=g_2 \cdot x
\end{equation}
for all $g_1,g_2 \in G$ and $x \in \mathcal{M}(G)$. 
Equation \ref{eqn3} holds for $x$ in the orbit of $x_0$ since for any $g \in G$, if $k \in K$ is such that $g_1k=g_2k$ then:
\begin{align*}
    g_2(gx_0) &= g_2kg(g^{-1}k^{-1}g)x_0\\
    &=g_2kg(x_0)\\
    &=g_1(gx_0)\\
\end{align*}
where the second equality uses that $K$ is normal. Then since the $G$-orbit of $x_0$ is dense in $\mathcal{M}(G)$, Equation \ref{eqn3} holds everywhere by continuity of the action. The action defined in \ref{eqn2} is easily seen to be continuous and $G/K$-equivariant and it is minimal by \ref{eqn3}. Now uniqueness of the universal minimal flow implies that the action of $H \simeq G/K$ on $\mathcal{M}(G)$ is the universal minimal flow of $H$. As the map $\mathcal{M}(G) \to \mathcal{M}(H)$ induced by $f$ is an $H$-flow morphism, it is in fact an $H$-flow isomorphism (see \cite{uspchains}, Proposition 3.3).
\end{proof}

It follows from Proposition \ref{prop4}, that the bottom two arrows on the diagram in Figure \ref{fig1} are $\Q^\times$-flow isomorphisms. In particular, we have that:

\begin{cor}\label{cor4}
The flows $\mathcal{M}(\aut(\bK))$ and $\mathcal{M}(\Homeo(K))$ are homeomorphic to $\mathcal{M}(\Q^\times)$. 
\end{cor}

\subsection{$\Homeo(K)$ is not a product}

We show now that $\Homeo(K)$ is not a direct product of $\Homeo^1(K)$ and $\Q^\times$.

\begin{lem}\label{lem6}
Every non-trivial conjugacy class in $\Homeo(K)$ is uncountable.
\end{lem}

\begin{proof}
Let $f \neq \textrm{id}$ be in $\Homeo(K)$. Let $K=\varprojlim(I_n, t_n^{n+1})$ where each $I_n=[0,1]$. Let $\mathbf{x}=(x_m)_{m\in\N}$ and $\mathbf{y}=(y_m)_{m\in\N}$ with $\mathbf{x} \neq \mathbf{y}$ and $f(\mathbf{x})=\mathbf{y}$. Let $n \in \N$ such that $x_n \neq y_n$. We will assume that $x_n <y_n$ (the other case is analogous). For each $b \in (x_n,1)$ let $g_b \in \Homeo_+[0,1]$ be such that $g_b(x_n)=x_n$ and $g_b(b)=y_n$. Each $g_b$ as a map from $I_n$ to $I_n$ induces a unique open continuous map $K \to K$ (the other coordinate maps are fully determined once $g_b:I_n\to I_n$ is set, see \cite{eberhart}, Lemma 3.3) which we denote by $\tilde{g_b}$. Each $\tilde{g_b}$ is a homeomorphism (see \cite{eberhart} top of p.129, where such maps are called \emph{vertically induced homeomorphisms}). One may check that $\tilde{g_b}(\mathbf{x})=\mathbf{x}$ and that there exists $\mathbf{z_b}=((z_b)_m)_{m\in\N} \in K$ with $(z_b)_n=b$ and $\tilde{g_b}(\mathbf{z_b})=\mathbf{y}$. So:
\[\tilde{g_b}^{-1} \circ f\circ \tilde{g_b}(\mathbf{x})=\mathbf{z_b}\]
In particular for $b \neq c$, since $\mathbf{z_b}\neq \mathbf{z_c}$ the equation above implies that $\tilde{g_b}^{-1}\circ f\circ \tilde{g_b} \neq \tilde{g_c}^{-1}\circ f\circ \tilde{g_c}$ and thus the conjugacy class of $f$ is uncountable. 
\end{proof}

\begin{prop}\label{prop5}
The group $\Homeo(K)$ is not isomorphic to the direct product $\Homeo^1 \times \Q^\times$.
\end{prop}

\begin{proof}
If $\Homeo(K) \simeq \Homeo^1(K) \times \Q^\times$, then $\Homeo(K)$ would contain a countably infinite normal subgroup. This is clearly impossibly by Lemma \ref{lem6}. 
\end{proof}

\section{Remarks on Ramsey degree}
An object $A$ in a projective Fraissé class $\mathcal{F}$ has \emph{infinite Ramsey degree} if for any $n \in \N$, there exists $B \in \mathcal{F}$ so that for any $C \in \mathcal{F}$ with $\hom(C,B) \neq \emptyset$ we have that: there exists a coloring $c:\hom(C,A) \to [n]$ such that for each $f \in \hom(C,B)$, $\left(\hom(B,A) \circ f \right)\cap c^{-1}(i) \neq \emptyset$ for each $i \in [n]$. Recall from Section \ref{sec6degree} that the group $\aut(\bK)$ has non-metrizable universal minimal flow. It is a theorem of Zucker that if $\mathcal{F}$ is a Fraissé class of rigid structures, then $\aut (\mathbb{F})$ has non-metrizable universal minimal flow if and only if $\mathcal{F}$ contains an object of infinite Ramsey degree (see \cite{zucker}, and for a proof of the dual of Zucker's theorem see \cite{dragan}). So we know abstractly that $\mathcal{K}$ must contain an object of infinite Ramsey degree. It is perhaps worth mentioning that one can prove this fact directly:

\begin{prop}
There is an object in $\mathcal{K}$ with infinite Ramsey degree.
\end{prop}

\begin{proof}
Let $A$ be a pointed linear graph with two vertices. Fix $n>1$. Let $B$ be a pointed linear graph with $2^{n+1}$ vertices. We claim that for any $C$ with $\hom(C,B)\neq \emptyset$ there is a $n$-coloring of $\Epi(C,A)$ so that for any $f \in \Epi (C,B)$, the set $\Epi(B,A) \circ f$ has morphisms colored by every one of the $n$ colors. 

For a natural number $k$, denote by $\rho(k)$ the largest natural number such that $2^{\rho(k)}\vert k$. Let $C$ be as above and then color $\Epi(C,A)$ as follows:

\[c(f) = \rho(\textrm{deg}(f)) \mod n\]

Since $\textrm{deg}(f\circ g)=\textrm{deg}(f)\textrm{deg}(g)$, it is easy to check that for any $f \in \Epi (C,B)$, 
\[\norm{c(\Epi(B,A) \circ f)} = n.\]
\end{proof}

We note that just the fact that $\aut(\bK)$ has non-metrizable universal minimal flow and is dense in $\Homeo (K)$ is not enough to conclude that $\Homeo(K)$ has non-metrizable universal minimal flow. For example, any countable dense subgroup of $\Homeo_+[0,1]$ (for instance, the group of all piece-wise linear homeomorphisms with finitely many pieces, breakpoints at rationals, and taking only rational values at breakpoints) has non-metrizable universal minimal flow whereas $\Homeo_+[0,1]$ is extremely amenable (Pestov, \cite{pestov}, Theorem 6.2).


\begin{thebibliography}{1}
\bibitem{bkgowers} D. Bartošová and A. Kwiatkowska, Gowers' Ramsey theorem with multiple operations and dynamics of the homeomorphism group of the Lelek fan. \emph{J. Combin. Theor Ser. A} {\bf 150} (2017), 108--136.

\bibitem{lelekfan} D. Bartošová and A. Kwiatkowska, Lelek fan from a projective Fraissé limit. \emph{Fund. Math.} {\bf 231} (2015), no.1, 57--79.

\bibitem{bkumf} D. Bartošová and A. Kwiatkowska, The universal minimal flow of the homeomorphism group of the Lelek fan. \emph{Trans. Amer. Math Soc.} {\bf 371} (2019), no. 10, 6995–7027. 

\bibitem{bymt} I. Ben Yaacov, J.  Melleray, and T. Tsankov. Metrizable universal minimal flows of Polish groups have a comeagre orbit. \emph{Geom. Funct. Anal.} {\bf 27} (2017), no. 1, 67--77.

\bibitem{char-roe} W.J. Charatonik and R.P. Roe, Projective Fraissé limits of trees. https://web.mst.edu/~rroe/Fraisse.pdf.

\bibitem{debski} W. D\k{e}bski, On topological types of the simplest indecomposable continuaa. \emph{Colloq. Math.} {\bf 49} (1985), no. 2, 203-211.

\bibitem{eberhartlattice} C. Eberhart, J.B. Fugate, S. Schumann, The lattice of Knaster continua. Proceedings of the First International Meeting on Continuum Theory (Puebla City, 2000) \emph{Topology Appl.} {\bf 126} (2002), no.3, 343-349.

\bibitem{eberhart} C. Eberhart, J. B. Fugate, and S. Schumann, Open maps between Knaster continua. \emph{Fund. Math.} {\bf 162} (1999), no.2, 119--148.

%\bibitem{glasnerweiss} E. Glasner, B. Weiss,  Minimal actions of the group $S(\mathbb{Z})$ of permutations of the integers. \emph{Geom. Funct. Anal.} {\bf 12} (2002), no. 5, 964–-988.

\bibitem{irw-sol} T. Irwin and S. Solecki, Projective Fraissé limits and the pseudo-arc. \emph{Trans. Amer. Math. Soc.} {\bf 358} (2006), no.7, 3077--3096.

\bibitem{jahelzucker} C. Jahel and A. Zucker, Topological dynamics of Polish group extensions. arXiv:1902.04901. 


\bibitem{kpt} A. S. Kechris, V. G. Pestov, and S. Todorcevic, Fraissé limits, Ramsey theory, and topological dynamics of automorphism groups. \emph{Geom. Funct. Anal.} {\bf 15} (2005), no. 1, 106--189.

\bibitem{krupski} P. Krupski, Solenoids and inverse limits of sequences of arcs with open bonding maps. \emph{Fund. Math.} {\bf 120} (1984), no.1, 41--52.

\bibitem{dragan} D. Mašulović, The Kechris-Pestov-Todorcevic correspondence from the point of view of category theory. \emph{Appl. Categ. Structures} {\bf 29} (2021), no.1, 141--169.

%\bibitem{nadler} Nadler, Sam B., Jr. Continuum theory: an introduction. \emph{Monographs and Textbooks in Pure and Applied Mathematics} {\bf 158} \emph{Marcel Dekker, Inc., New York} (1992).

\bibitem{pan-sol} A. Panagiotopoulos and S. Solecki, A combinatorial model for the Menger curve. \emph{J. Topol. Anal.} {\bf 14} (2022), no.1, 203--229.

\bibitem{pestovbook} V. Pestov, Dynamics of infinite-dimensional groups: the Ramsey-Dvoretzky-Milman phenomenon, \emph{AMS University Lecture Series} {\bf 40}, \emph{American Mathematical Society, Providence, RI} (2006).

\bibitem{pestov} V. Pestov, On free actions, minimal flows, and a problem by Ellis. \emph{Trans. Amer. Math. Soc.} {\bf 350} (1998), no. 10, 4149--4165.

\bibitem{uspchains} V. Uspenskij, On universal minimal compact $G$-spaces. \emph{Proeceedings of the 2000 Topology and Dynamics Conference,} \emph{Topology Proc.} {\bf 25} (2000), 301--308.

\bibitem{watkins} W.T. Watkins, Homeomorphic classification of certain inverse limit spaces with open bonding maps. \emph{Pacific J. Math.} {\bf 103} (1982), no. 2, 589--601.

\bibitem{wickman} L. Wickman, Projective Fraissé theory and Knaster continua, Ph.D. Thesis, University of Florida, 2022.

\bibitem{zucker} A. Zucker, Topological dynamics of automorphism groups, ultrafilter combinatorics, and the generic point problem. \emph{Trans. Amer. Math. Soc} \textbf{368} (2016), no. 9, 6715-6740.
\end{thebibliography}
\end{document}